\documentclass[11pt]{article}
\usepackage[letterpaper,top=1in,bottom=1in,left=1in,right=1in,marginparwidth=1.75cm]{geometry}

\usepackage{amsfonts}
\usepackage{mathtools,amssymb, amsthm}

\usepackage[most]{tcolorbox}
\usepackage{graphicx}
\usepackage[colorlinks=true, allcolors=blue]{hyperref}
\usepackage{cleveref}
\usepackage{algorithm}
\usepackage{algpseudocode}
\usepackage{todonotes}
\usepackage{lscape}
\usepackage{booktabs}
\usepackage{authblk}
\usepackage{pdflscape}
\usepackage{makecell}
\usepackage{utfsym}

\newcommand\keywordsname{Key words}
\usepackage{hyperref}
\usepackage{caption}
\usepackage{subcaption}

\DeclareMathOperator*{\argmin}{arg\,min}

\newcommand{\R}{\mathbb{R}}

\newtheorem{theorem}{Theorem}[section]

\newtheorem{lemma}{Lemma}[section]

\newtheorem{remark}{Remark}[section]

\usepackage{tabularx,booktabs}
\usepackage{color}
\usepackage{lscape}
\usepackage{latexsym}
\usepackage{rotating}
\usepackage{amsmath,amsfonts,amssymb,enumerate}
\usepackage{threeparttable}
\usepackage{graphicx}
\usepackage{multirow}
\usepackage{algpseudocode,algorithm}
\usepackage{bm}
\counterwithout{equation}{section}
\DeclareMathOperator*{\prox}{prox}

\DeclareMathOperator*{\dom}{dom}
\newcommand{\mtrx}[1]{\mathsf{#1}}
\newcommand{\mA}{\mtrx{A}}
\newcommand{\mB}{\mtrx{B}}
\newcommand{\mD}{\mtrx{D}}

\newcommand{\mI}{\mtrx{\mathrm{I}}}

\newcommand{\mP}{\mtrx{P}}
\newcommand{\mQ}{\mtrx{Q}}

\newcommand{\vect}[1]{\bm{#1}}
\newcommand{\va}{\vect{a}}
\newcommand{\vb}{\vect{b}}

\newcommand{\vu}{\vect{u}}
\newcommand{\vv}{\vect{v}}
\newcommand{\vx}{\vect{x}}
\newcommand{\vy}{\vect{y}}
\newcommand{\vz}{\vect{z}}
\newcommand{\vw}{\vect{w}}

\newcommand{\vzero}{\vect{0}}

\usepackage{enumitem, xcolor, arydshln, url}
\usepackage{multirow}
\usepackage{caption}
\usepackage{subcaption}
\newcommand{\N}{\mathbb{N}}

\newcommand{\la}{\langle}
\newcommand{\ra}{\rangle}

\usepackage{algorithm}
\usepackage{algpseudocode}

\title{Signal and Image Recovery with Scale and Signed Permutation Invariant Sparsity-Promoting Functions}

\author[1]{Jianqing Jia$^\ast$}
\author[2]{Ashley Prater-Bennette}
\author[3]{Lixin Shen$^\ast$}

\affil[1]{School of Data Science and Society, University of North Carolina at Chapel Hill, Chapel
Hill, NC 27599, USA}
\affil[2]{Air Force Research Laboratory, Rome, NY 13441, USA}
\affil[3]{Department of Mathematics, Syracuse University, Syracuse, NY 13244, USA}

\makeatletter
\def\blfootnote{\gdef\@thefnmark{}\@footnotetext}
\makeatother
\newenvironment{@abssec}[1]{
     \if@twocolumn
       \section*{#1}
     \else
       \vspace{.05in}\footnotesize
       \parindent .0in
         {\upshape\bfseries #1. }\ignorespaces
     \fi}
     {\if@twocolumn\else\par\vspace{.1in}\fi}
\newenvironment{keywords}{\begin{@abssec}{\keywordsname}}{\end{@abssec}}
\date{}  

\begin{document}

\maketitle

\begin{abstract}
Sparse signal recovery has been a cornerstone of advancements in data processing and imaging. Recently, the squared ratio of $\ell_1$ to $\ell_2$ norms, $(\ell_1/\ell_2)^2$, has been introduced as a sparsity-prompting function, showing superior performance compared to traditional $\ell_1$ minimization, particularly in challenging scenarios with high coherence and dynamic range.  This paper explores the integration of the proximity operator of  $(\ell_1/\ell_2)^2$ {and $\ell_1/\ell_2$} into efficient optimization frameworks, including the Accelerated Proximal Gradient (APG) and Alternating Direction Method of Multipliers (ADMM). We rigorously analyze the convergence properties of these algorithms and demonstrate their effectiveness in compressed sensing and image restoration applications. Numerical experiments highlight the advantages of our proposed methods in terms of recovery accuracy and computational efficiency, particularly under noise and high-coherence conditions. 
 \end{abstract}

\blfootnote{$^\ast$Equal Contribution. Emails: \texttt{jqjia@unc.edu}, \texttt{ashley.prater-bennette@us.af.mil}, \texttt{lshen03@syr.edu}. Submitted on Dec 27, 2024; accepted on July 19, 2025.}

\normalsize
\section{Introduction}
In the era of big data, sparse learning has emerged as a transformative paradigm, enabling efficient representation and processing of inherently sparse datasets. Sparse modeling techniques have proven particularly valuable in fields such as signal processing, compressed sensing, and image restoration, where the goal is often to recover high-dimensional signals with minimal active components.  

Among the array of sparsity-promoting functions, the ratio of $\ell_1$ to $\ell_2$ norms, specifically its squared form $(\ell_1/\ell_2)^2$, has gained attention for its robustness in handling high-coherence and dynamic range scenarios. Unlike traditional minimization $\ell_1$, this approach dynamically balances sparsity and energy, resulting in superior recovery performance in complex environments. Despite its potential, integrating this function into an optimization framework has been non-trivial, primarily due to its nonconvexity and computational challenges. 

This paper builds on the foundation work in \cite{jia2024computing,jia2025sparse}, where the proximity operator of $(\ell_1/\ell_2)^2$ {and $\ell_1/\ell_2$ were} efficiently computed. Leveraging this breakthrough, we integrate the proximity operator into two efficient optimization methods: the Accelerated Proximal Gradient (APG) and the Alternating Direction Method of Multipliers (ADMM). Our contributions are threefold:
\begin{itemize}
    \item[1.] We present novel formulations for incorporating $(\ell_1/\ell_2)^2$ into APG and ADMM frameworks.
    \item[2.] We rigorously analyze the convergence properties of these methods under various setting.
    \item[3.] We conduct extensive numerical experiments demonstrating the superior performance of our proposed algorithms in sparse signal recovery and image restoration tasks. 
    
\end{itemize}
The remainder of this paper is organized as follows: Section 2 provides the necessary mathematical preliminaries and background. Section 3 elaborates on the application of the proposed methods to sparse signal recovery, including algorithm details and theoretical analysis. Section 4 extends the discussion to image restoration, showcasing practical implementations and results. Finally, Section 5 concludes the paper by summarizing key findings and highlighting future research directions.

\section{Preliminaries}

All functions in this work are defined on Euclidean space $\mathbb{R}^n$. Bold lowercase letters, such as $\vx$, signify vectors, with the $j$th component represented by the corresponding lowercase letter $x_j$. The notation $\mathrm{supp}(\vx)$ denotes the support of the vector $\vx$, defined as $\mathrm{supp}(\vx)=\{k: x_k \neq 0\}$. Matrices are indicated by bold uppercase letters such as $\mA$ and $\mB$. Let $\mathcal{P}_n$ denote the set of all $n\times n$ signed permutation matrices: those matrices that have only one nonzero entry in every row or column, which is $\pm1$.

The $\ell_p$ norm of $\vx=[x_1,\ldots, x_n]^\top \in \mathbb{R}^n$ is defined as $\|\vx\|_p=(\sum_{k=1}^n |x_k|^p)^{1/p}$ for $1\le p <\infty$, $\|\vx\|_\infty=\max_{1\le k \le n} |x_k|$,  and $\|\vx\|_0$ being the number of non-zero components in $\vx$. The Frobenius norm of matrix $\mA$ is defined as $\|\mA\|_F=(\sum_{i=1}^m \sum_{j=1}^n |a_{ij}|^2)$. The standard inner product in $\mathbb{R}^n$ is denoted by $\langle \vu,\vv\rangle$, where $\vu$ and $\vv$ are vectors in $\mathbb{R}^n$. 

A function $f$ defined on $\mathbb{R}^n$ with values in $\mathbb{R} \cup \{+\infty\}$ is proper if its domain $\mathrm{dom}(f)=\{x\in \mathbb{R}^n: f(x)<+\infty\}$ is nonempty, and $f$ is lower semicontinuous if its epigraph $\{(\vx,t)\in \mathbb{R}^n \times \mathbb{R}: f(\vx)\le t\}$ is a closed set. The set of proper and lower semicontinuous functions on  $\mathbb{R}^n$ to $\mathbb{R} \cup \{+\infty\}$ is denoted by $\Gamma(\mathbb{R}^n)$. 

We focus on two important properties of functions:
\begin{itemize}
\item[1.] Scale invariance: A function $f: \mathbb{R}^n \rightarrow \mathbb{R}$  is considered scale invariant if for any $\vx\in \mathbb{R}^n$ and any positive scalar $\alpha$, it holds that: 
$$
f(\alpha \vx) = f(\vx). 
$$
This means the function value remains unchanged when its input is scaled by any positive constant.

\item[2.] Signed permutation invariance: A function $f$ is signed permutation invariant if, for all permutations $\mP \in \mathcal{P}_n$ and all vectors $\vx \in \mathbb{R}^n$, the following condition is satisfied:
$$
f(\mP\vx) = f(\vx). 
$$
\end{itemize}

Examples of these invariances include all \( \ell_p \) norms (where \( 0 \leq p \leq \infty \)) and the log-sum penalty function in \( \mathbb{R}^n \), which are signed permutation invariant but not scale invariant. Conversely, the \( \ell_0 \) norm and the effective sparsity measure \(\left(\frac{\|\cdot\|_q}{\|\cdot\|_1}\right)^{\frac{q}{1-q}}\), where \( q \) is in \( (0, \infty) \setminus \{1\} \) as discussed in \cite{lopes2013estimating}, exemplify functions that are both scale and signed permutation invariant.

In this paper, from the class of scale and signed permutation invariant functions, we focus on the following specific functions:
\begin{equation*}
h_p(\vx)=\left\{
\begin{array}{ll}
    \left(\frac{\|\vx\|_1}{\|\vx\|_2}\right)^p, & \hbox{if $\vx \neq \mathbf{0}$;} \\
    0,&{otherwise},
  \end{array}
\right.
\end{equation*}
for $p=1$ and $2$. The algorithms for computing the proximity operator of $h_2$ and $h_1$ are detailed in our previous work \cite[Algorithm 1]{jia2024computing} and \cite[Algorithm 2]{jia2024computing}, respectively. 

The proximity operator is a mathematical concept used in optimization. Given a proper lower semicontinuous function $f$ and a point $\vx$, the proximity operator of $f$ at $\vx\in \R^n$ with index $\lambda$, denoted as $\mathrm{prox}_{\lambda f}(\vx)$, is defined as:
\begin{equation}\label{def:prox}
    \mathrm{prox}_{\lambda f} (\vx) := \mathrm{arg}\min \left\{\frac{1}{2\lambda} \|\vx-\vu\|_2^2+f(\vu): \vu\in\mathbb{R}^n\right\}.
\end{equation}
In simpler terms, the proximity operator finds a point $\vu$ that minimizes the sum of the function $f$ and half the squared Euclidean distance between $\vu$ and a given point $\vx$.  The proximity operator provides a computationally efficient way to find solutions for optimization problems involving nonsmooth functions.

\section{Application: Sparse Signal Recovery}
Sparse signal recovery involves optimization problems that balance data fidelity with sparsity promotion. In this section, we focus on the following unconstrained optimization problem:
\begin{equation} \label{model:uncons}
    \min \left\{  \lambda h_p(\vx) + \frac{1}{2} \|\mA\vx - \vb\|_2^2\,:\,\vx\in \R^n \right\},
\end{equation}
where $h_p(\vx)$ promotes sparsity of $\vx$, $\lambda>0$ is a regularization parameter, and $\mA$ and $\vb$ represent the sensing matrix and measurements, respectively. Unlike constrained models that strictly enforce $\mA\vx=\vb$, the unconstrained problem~\eqref{model:uncons} accommodates both noisy and noiseless data. 

Due to the nonconvex nature of the function $h_p$, problem~\eqref{model:uncons} requires advanced optimization methods. Common frameworks for nonconvex, nonsmooth problems of the form:
\begin{equation}\label{eqn:f+g}
    \min_{\vx\in \R^n} \quad f(\vx)+g(\vx)
\end{equation}
includes Proximal Gradient method (forward-backward splitting, FBS)\cite{attouch2013convergence}, Accelerated Proximal Gradient method (APG)\cite{Beck-Teboulle:SIAMIS:09,Li-Lin:NISP:2015}, Inertial Forward-Backward (IFB), Inertial Proximal Algorithm for Nonconvex Optimization (iPiano) \cite{Ochs-Chen-Brox-Pock:SIAMIS:2014}, General Iterative Shrinkage and Thresholding (GIST) and Gradient Descent with Proximal Average (GDPA). GIST and GDPA require that $g$ should be explicitly written as a difference of two convex functions to have the convergence guarantee, and iPiano demands the convexity of $g$. These methods are not ideal for our problem \eqref{model:uncons}. Thus, in this section, we mainly explore the integration of the proximity operator of $h_p$ (for $p=1, 2$) into FBS, APG, and ADMM, analyze the convergence of these algorithms, and use numerical experiments to demonstrate the effectiveness of the methods.

\subsection{Algorithms: FBS and APG}

\textbf{Forward-backward Splitting (FBS).}
In addressing the challenges posed by the nonconvex setting, we employ the forward-backward splitting method (also known as the proximal gradient method)\cite{attouch2013convergence}. Each iteration of FBS is defined as:
\begin{equation*}
    \vx^{k+1} \in \text{prox}_{\gamma^k \lambda h_p} \left( \vx^{k} - \gamma^k \mA^{\top}(\mA\vx^k-\vb) \right),
\end{equation*}
where $\gamma^k$ represents the step size at iteration $k$, chosen adaptively within the range: 
$$
0 < \underline{\gamma} < \gamma^k < \overline{\gamma} < \frac{1}{\|\mA^{\top}\mA\|_F}.
$$ 
This step size ensures stability and effectiveness during optimization \cite{attouch2013convergence}. 

Before showing the convergence of the algorithm, let us first review the definition of the Kurdyka-\L ojasiewicz (KL) property of a function.

Let $f:\mathbb{R}^d\to (-\infty,+\infty]$ be proper and lower semicontinuous.  We note that $[f<\mu]:=\{\vx\in \mathbb{R}^n: f(\vx)<\mu\}$ and $[\eta<f<\mu]:=\{\vx\in \mathbb{R}^n: \eta<f(\vx)<\mu\}$. Let $r_0>0$ and set 
$$
\mathcal{K}(r_0):=\{\varphi \in C^0([0,r_0))\cap C^1((0,r_0)), \; \varphi(0)=0, \; \varphi\; \mbox{is concave and}\; \varphi' >0 \}. 
$$
The function $f$ satisfies the KL inequality (or has KL property) locally  at $\tilde{\vx}\in \dom \partial f$ if there exist $r_0>0$, $\varphi \in \mathcal{K}(r_0)$ and a neighborhood $U(\tilde{\vx})$ of $\tilde{\vx}$ such that
\begin{equation}\label{KL-inequality}
\varphi'(f(\vx)-f(\tilde{\vx}))\text{dist}(\mathbf{0},\partial f(\vx))\geq 1
\end{equation}
for all $\vx \in U(\tilde{\vx}) \cap [f(\tilde{\vx})<f(\vx)<f(\tilde{\vx})+r_0]$. The function $f$ has the KL property on $S$ if it does so at each point of $S$.

Since $\|\cdot\|_1^2$ and $\|\cdot\|_2^2$ are semialgebraic, so is $\lambda h_p(\vx) + \frac{1}{2} \|\mA\vx - \vb\|_2^2$. Therefore, the objective function of \eqref{model:uncons} has the KL property on its domain, see \cite{attouch2013convergence}. Also, the objective function of \eqref{model:uncons} is nonnegative and coercive. Then the sequence $\{\vx^k\}$ in the algorithm is bounded and converges to some critical point of problem \eqref{model:uncons} (\cite[Theorem 5.1]{attouch2013convergence}). Moreover, the sequence $\{\vx^k\}$ has a finite length, i.e. $\sum_{k=1}^\infty \|\vx^{k+1}-\vx^k\|_2<\infty$. 

The proximal operator $\prox_{\gamma^k \lambda h_2}$ can be computed by \cite[Algorithm 1]{jia2024computing} and $\prox_{\gamma^k \lambda h_1}$ can be computed using the iterative algorithm presented by \cite[Algorithm 2]{jia2024computing}, corresponding to $p=2$ and $p=1$, respectively. For $p=2$, $\prox_{\gamma^k \lambda h_2}$ has a closed-form solution, making it highly efficient to compute. In contrast, for $p=1$, the computation relies on an iterative procedure, making it more computationally demanding. In this paper, we primarily focus on the $p=2$ case in the experimental section. We describe the signal recovery algorithm for problem \eqref{model:uncons} in Algorithm~\ref{alg:FBS_CS}, denoted as \textbf{H2-FBS}.

\begin{algorithm} 
\caption{\textbf{H2-FBS}: The forward-backward splitting algorithm for problem \eqref{model:uncons}} \label{alg:FBS_CS}
\begin{algorithmic}[1] 
\State \textbf{Input:} Initialize $\vx^0 \in \mathbb{R}^n$, $0<\gamma<\frac{1}{\|\mA^\top\mA\|_F}$
    \For{$k=0,1,2,\cdots$}
        \State\begin{equation*}
            \vx^{k+1}\in \text{prox}_{\gamma \lambda h_p} \left( \vx^{k} - \gamma \mA^{\top}(\mA\vx^k-\vb) \right)
        \end{equation*}
    \EndFor
\State \textbf{Output:} $\vx^\infty$
\end{algorithmic}
\end{algorithm}

\textbf{Accelerated Proximal Gradient (APG) Method.}
We further consider an accelerated proximal gradient method \cite{Li-Lin:NISP:2015} to speed up the convergence of the forward-backward splitting method. In particular, the algorithm is described in Algorithm~\ref{alg:APG_CS} and we denote the algorithm as \textbf{H2-APG}.

\begin{algorithm} 
\caption{\textbf{H2-APG}: The accelerated proximal gradient algorithm for problem \eqref{model:uncons}} \label{alg:APG_CS}
\begin{algorithmic}[1] 
\State \textbf{Input:} Initialize $\vz^1=\vx^1=\vx^0 \in \mathbb{R}^n$, $t^1=1$, $t^0=0$, $\alpha_x<1/\|\mA\|^2$, $\alpha_y<1/\|\mA\|^2$, and $F(\vx)$ is the objective function of problem \eqref{model:uncons}
    \For{$k=1,2,\cdots$}
        \State \begin{eqnarray}
    \vy^k & = & \vx^k + \frac{t^{k-1}}{t^k}(\vz^k - \vx^k) + \frac{t^{k-1}-1}{t^k}(\vx^k - \vx^{k-1})\notag\\
    \vz^{k+1} & = & \text{prox}_{\alpha_y \lambda h_p}\left(\vy^{k} - \alpha_y \mA^{\top}(\mA\vy^k-\vb)\right)\notag\\
    \vv^{k+1} & = & \text{prox}_{\alpha_x \lambda h_p}\left(\vx^{k} - \alpha_x \mA^{\top}(\mA\vx^k-\vb)\right)\notag\\
    t^{k+1} & = & \frac{\sqrt{4(t^{k})^2+1}+1}{2}\notag\\
    \vx^{k+1} & = & \left\{
                 \begin{array}{ll}
                   \vz^{k+1}, & \hbox{if $F(\vz^{k+1})\le F(\vv^{k+1})$;} \\
                   \vx^{k+1}, & \hbox{otherwise.} 
                 \end{array}
               \right.\notag
\end{eqnarray}
    \EndFor
\State \textbf{Output:} $\vx^\infty$
\end{algorithmic}
\end{algorithm}
The algorithm performs well when $\alpha_x$ and $\alpha_y$ are both set to the same value, denoted as $\gamma$. It was shown in \cite{Li-Lin:NISP:2015} that the algorithm converges to a critical point.  Let us recall some concepts and establish supporting results below.

A extended real-valued function $f:\mathbb{R}^n \rightarrow (-\infty, \infty]$ is said to be proper if its domain $\mathrm{dom} f:=\{\vx: f(\vx)<\infty\}$. Additionally, a proper function $f$ is said to be closed if it is lower semi-continuous. For a proper closed function $f$, the regular subdifferential $\hat{\partial} f(\bar{\vx})$  and the limiting subdifferential $\partial f(\bar{\vx})$ at $\bar{\vx}\in \mathrm{dom} f$   are given respectively as 
\begin{eqnarray*}
    \hat{\partial} f(\bar{\vx})&:=&\left\{\vv: \lim_{\vx \rightarrow \bar{\vx}}\inf_{\vx \neq \bar{\vx}} \frac{f(\vx)-f(\Bar{\vx})-\langle \vv, \vx-\Bar{\vx}\rangle}{\|\vx-\bar{\vx}\|_2}\ge 0 \right\},\\
    \partial f(\bar{\vx})&:=&\left\{\vv: \exists  \vx^{(t)} \stackrel{f}{\rightarrow} \bar{\vx}\; \mbox{and}\; \vv^{(t)}\in \hat{\partial}f(\vx^{(t)}) \; \mbox{with} \; \vv^{(t)} \rightarrow \vv \right\}, 
\end{eqnarray*}
where $\vx^{(t)} \stackrel{f}{\rightarrow} \bar{\vx}$ means $\vx^{(t)} \rightarrow \vx$ and $f(\vx^{(t)}) \rightarrow f(\vx)$. For a proper closed function $f$, we assert that $\bar{\vx}$ is a stationary point of $f$ is $0 \in \partial f(\bar{\vx})$. 

\begin{lemma}[Theorem 1 in \cite{Li-Lin:NISP:2015}]
Let $f$ be a proper function with Lipschitz continuous gradients and $g$ be proper and lower semicontinuous. For nonconvex $f$ and nonconvex nonsmooth $g$, assume that $F(\vx)=f(\vx)+ g(\vx)$ coercive. Then $\{\vx^k\}$ and $\{\vv^k\}$ generated by Algorithm~\ref{alg:APG_CS} are bounded. Let $\vx^\star$ be any accumulation point of $\{\vx^k\}$, we have $0\in \partial F(\vx^\star)$, i.e. $\vx^\star$ is a critical point.
\end{lemma}

The proof can be found in the supplementary materials of \cite{Li-Lin:NISP:2015}.

\subsection{Algorithm: Alternating Direction Method of Multipliers (ADMM)}

For problems with separable structures, such as sparse signal recovery, the alternating direction method of multipliers (ADMM) is particularly effective. By introducing  an auxiliary variable $\vy$, the optimization problems~\eqref{model:uncons} can be reformulated as:
\begin{align} \label{model:cons_y}
    \min_{\vx, \vy \in \R^n} &\quad \lambda h_p(\vx) + \frac{1}{2}\|\mA\vy - \vb\|_2^2\\
    s.t. & \quad \vx = \vy \nonumber
\end{align}
We define the augmented Lagrangian of \eqref{model:cons_y}
as
\begin{equation} \label{eqn:Lagrangian_CS}
    \mathcal{L}(\vx, \vy, \vz) = \lambda h_p(\vx) + \frac{1}{2}\|\mA\vy - \vb\|_2^2 + \la \vz, \vx - \vy \ra + \frac{\rho}{2}\|\vx - \vy\|_2^2, 
\end{equation}
where $\vz$ is the Lagrangian multiplier (dual variable) and $\rho>0$ is a parameter. The ADMM scheme iterates as follows:
\begin{subequations} \label{eqn:ADMMsc}
\begin{align}
        \vx^{k+1} &\in \argmin_{\vx\in \R^n} \left(\lambda h_p(\vx)+ \frac{\rho}{2} \left\|\vx - \vy^k + \frac{\vz^k}{\rho}\right\|_2^2\right)\label{eqn:ADMMsc_a}\\
        \vy^{k+1} &= \argmin_{\vy\in \R^n} \left( \frac{1}{2}\|\mA\vy - \vb\|_2^2+ \frac{\rho}{2}\left\| \vy - \vx^{k+1} - \frac{\vz^k}{\rho}\right\|_2^2 \right)\label{eqn:ADMMsc_b}\\ 
        \vz^{k+1} &= \vz^k + \rho (\vx^{k+1} - \vy^{k+1})   \label{eqn:ADMMsc_c} 
\end{align}
\end{subequations}

The $\vy$-subproblem has a closed-form solution given by
$$
\vy^{k+1} = \left(\frac{1}{\rho}\mA^\top \mA +\mI_n \right)^{-1} \left( \frac{\mA^\top \vb}{\rho} + \frac{\vz^k}{\rho} + \vx^{k+1} \right).
$$
Since we often encounter an underdetermined system in compressed sensing, we apply the Sherman-Morrison-Woodburg theorem, leading to a more efficient $\vy$-update
$$
\vy^{k+1} = \left(\mI_n - \frac{1}{\rho} \mA^\top(\mI_m + \frac{1}{\rho}\mA\mA^\top)^{-1}\mA \right) \left( \frac{\mA^\top \vb}{\rho} + \frac{\vz^k}{\rho} + \vx^{k+1} \right)
$$
as $\mA\mA^\top$ has a smaller dimension than $\mA^\top \mA$.

The $\vx$-subproblem could be based on the proximal operator $\prox_{\frac{1}{\rho} h_p}(\vx)$. Specifically, by letting $\vx =\vy^k - \vz^k/\rho$ and $\rho = \rho/\lambda$, we observe that the $\vx$-subproblem coincides with $\prox_{\frac{1}{\rho} h_p}(\vx)$, which can be computed by \cite[Algorithm 1]{jia2024computing} and \cite[Algorithm 2]{jia2024computing}. The algorithm for solving problem \eqref{model:uncons} is described as Algorithm~\ref{alg:ADMM_CS} and we denote the algorithm as \textbf{H2-ADMM}.

\begin{algorithm} 
\caption{\textbf{H2-ADMM}: The ADMM algorithm for problem \eqref{model:uncons}} \label{alg:ADMM_CS}
\begin{algorithmic}[1] 
\State \textbf{Input:} Initialize $\vx^0 \in \mathbb{R}^n$
    \For{$k=1,2,\cdots$}
        \State \begin{eqnarray}
    \vx^{k+1} & \in & \text{prox}_{\frac{\lambda}{\rho}h_p}  \left(\vy^k-\frac{\vz^k}{\rho}\right)\notag\\
    \vy^{k+1} &=& \left(\mI_n - \frac{1}{\rho} \mA^\top(\mI_m + \frac{1}{\rho}\mA\mA^\top)\mA \right)^{-1} \left( \frac{\mA^\top \vb}{\rho} + \frac{\vz^k}{\rho} + \vx^{k+1} \right)\notag\\
    \vz^{k+1} &=& \vz^k + \rho (\vx^{k+1} - \vy^{k+1})\notag
\end{eqnarray}
    \EndFor
\State \textbf{Output:} $\vx^\infty$
\end{algorithmic}
\end{algorithm}

Following the convergence analysis in \cite{wang2019global}, we firstly prove that the augmented Lagrangian function \eqref{eqn:Lagrangian_CS} is sufficiently decreasing.

\begin{lemma} \label{lem:ADMM_CS_conv}
    Let $\{\vx^k, \vy^k, \vz^k\}$ be the sequence generated by \eqref{eqn:ADMMsc}, and then we have 
    \begin{equation}\label{eqn:ADMMscobj_decreasing}
        \mathcal{L}(\vx^{k+1}, \vy^{k+1}, \vz^{k+1})\le \mathcal{L}(\vx^{k}, \vy^{k}, \vz^{k})- \left(\frac{\rho}{2}- \frac{L^2}{\rho}\right)\|\vy^k - \vy^{k+1}\|_2^2,
    \end{equation}
    where $L:=\sigma(\mA^{\top}\mA)$ is the largest eigenvalue of $\mA^{\top}\mA$.
\end{lemma}

\begin{proof}
Denote $\Phi(\vy):=\frac{1}{2}\|\mA\vy - \vb\|_2^2$. We know $\Phi$ is convex, which leads to 
\begin{equation}
    \Phi(\vy^k) \ge \Phi(\vy^{k+1}) - \la \vy^{k+1} - \vy^k, \nabla\Phi(\vy^{k+1})\ra.
\end{equation}
From the optimality condition of \eqref{eqn:ADMMsc_b} we have $\nabla\Phi(\vy^{k+1})= \vz^{k+1}$ and $\nabla\Phi(\vy^{k})= \vz^{k}$. Also $\nabla\Phi$ is Lipschitz continuous with constant $L$, which implies
\begin{equation}\label{eqn:ADMMsc_gradientLip}
    \|\vz^k - \vz^{k+1}\|_2 = \|\nabla\Phi(\vy^{k}) - \nabla\Phi(\vy^{k+1})\|_2 \le L\|\vy^k - \vy^{k+1}\|_2.
\end{equation}

Now we calculate
\begin{align}
    &\mathcal{L}(\vx^{k+1}, \vy^k, \vz^k) - \mathcal{L}(\vx^{k+1}, \vy^{k+1}, \vz^{k+1})\notag\\
    =&\Phi(\vy^k)-\Phi(\vy^{k+1})+\la \vz^k, \vx^{k+1} - \vy^k \ra -\la \vz^{k+1}, \vx^{k+1} - \vy^{k+1} \ra+ \frac{\rho}{2}\|\vx^{k+1}-\vy^k\|_2^2 - \frac{\rho}{2}\|\vx^{k+1}-\vy^{k+1}\|_2^2 \notag\\
    =&\Phi(\vy^k)-\Phi(\vy^{k+1})-\la \vz^k, \vy^k - \vy^{k+1} \ra - \rho \|\vx^{k+1} - \vy^{k+1}\|_2^2+ \frac{\rho}{2}\|\vy^{k+1}-\vy^k\|_2^2 + \rho \la \vx^{k+1} -\vy^{k+1}, \vy^{k}-\vy^{k+1} \ra \notag\\
    =& \Phi(\vy^k)-\Phi(\vy^{k+1}) - \la \vz^{k+1}, \vy^k - \vy^{k+1}\ra - \rho \la \vx^{k+1} - \vy^{k+1}, \vx^{k+1}- \vy^{k}\ra\notag\\
    &+\frac{\rho}{2}\|\vy^{k+1}-\vy^k\|_2^2 + \rho \la \vx^{k+1} -\vy^{k+1}, \vy^{k}-\vy^{k+1} \ra\notag\\
    =&\Phi(\vy^k)-\Phi(\vy^{k+1}) -  \la \vz^{k+1}, \vy^k - \vy^{k+1}\ra + \frac{\rho}{2}\|\vy^{k+1}-\vy^k\|_2^2 -\frac{1}{\rho}\|\vz^{k+1} - \vz^{k}\|_2^2\notag\\
    \ge & \frac{\rho}{2}\|\vy^{k+1}-\vy^k\|_2^2-\frac{1}{\rho}\|\vz^{k+1} - \vz^{k}\|_2^2\notag\\
    \ge& \left(\frac{\rho}{2}- \frac{L^2}{\rho}\right)\|\vy^{k+1} - \vy^k\|_2^2
\end{align}
which means
$$
\mathcal{L}(\vx^{k+1}, \vy^{k+1}, \vz^{k+1})\le \mathcal{L}(\vx^{k+1}, \vy^k, \vz^k) - \left(\frac{\rho}{2}- \frac{L^2}{\rho}\right)\|\vy^{k+1} - \vy^k\|_2^2.
$$

By \eqref{eqn:ADMMsc_a} we have
$$
\mathcal{L}(\vx^{k+1},\vy^k, \vz^k)\le \mathcal{L}(\vx^{k},\vy^k, \vz^k),
$$
which leads to
$$
\mathcal{L}(\vx^{k+1}, \vy^{k+1}, \vz^{k+1})\le \mathcal{L}(\vx^{k}, \vy^{k}, \vz^{k})- \left(\frac{\rho}{2}- \frac{L^2}{\rho}\right)\|\vy^k - \vy^{k+1}\|_2^2.
$$
This completes the proof.

\end{proof}

The next result provides the subsequential convergence analysis.
\begin{theorem}
    Let $\{\vx^k, \vy^k, \vz^k\}$ be the sequence generated by \eqref{eqn:ADMMsc}. If $\{\vx^k\}$ is bounded and $\rho> \sqrt{2}L$, then $\{\vx^k, \vy^k, \vz^k\}$ is bounded and has at least one accumulation point, and any accumulation point $(\vx^\star, \vy^\star, \vz^\star)$ is a stationary point of \eqref{eqn:Lagrangian_CS}. Additionaly, $\vx^\star$ is a stationary point of $\lambda h_p(\vx) + \frac{1}{2}\|\mA\vy - \vb\|_2^2$.
\end{theorem}

\begin{proof}
    Under the assumption $\rho> \sqrt{2}L$, from Lemma~\ref{lem:ADMM_CS_conv} we know 
    $$
    \mathcal{L}(\vx^{k}, \vy^{k}, \vz^{k})\le \mathcal{L}(\vx^{0}, \vy^{0}, \vz^{0}).
    $$
    Thus $\mathcal{L}(\vx^{k}, \vy^{k}, \vz^{k})$ is bounded. Also, we have
    \begin{equation*}
        \mathcal{L}(\vx^{k}, \vy^{k}, \vz^{k}) = \lambda h_p(\vx^k)+\Phi(\vy^k) + \frac{\rho}{2}\|\vx^k - \vy^k\|_2^2 + \la \nabla\Phi(\vy^k), \vx^k - \vy^k \ra,
    \end{equation*}
and utilizing convexity of $\Phi$ we get
$$
\Phi(\vx^k)\le \Phi(\vy^k)+ \la \nabla\Phi(\vy^k), \vx^k - \vy^k \ra + \frac{L}{2}\|\vx^k - \vy^k\|_2^2.
$$
Substituting back we have 

$$
         \mathcal{L}(\vx^{k}, \vy^{k}, \vz^{k}) \ge \lambda h_p(\vx^k) + \Phi(\vx^k)+ \frac{\rho-L}{2}\|\vx^k-\vy^k\|_2^2 \ge 0.
$$
    If $\{\vx^k\}$ is bounded and $\rho>\sqrt{2}L>L$, then $\|\vx^k - \vy^k\|$ is bounded, leading to the boundedness of $\{\vy^k\}$.
Also, since
\begin{equation*}
    \|\vz^k\|_2 = \|\nabla \Phi(\vy^k) - \nabla\Phi(\vzero) + \nabla \Phi(\vzero)\|_2 \le L \|\vy^k\|_2 + \|\nabla\Phi(\vzero)\|_2,
\end{equation*}
we have the sequence of $\{\vx^k, \vy^k, \vz^k\}$ is bounded, hence it has at least one accumulation point.

Now, by summing \eqref{eqn:ADMMscobj_decreasing} from $k=0$ to $\infty$ we can get
$$
\sum_{k=0}^{\infty} \left[ \mathcal{L}(\vx^{k}, \vy^{k}, \vz^{k}) - \mathcal{L}(\vx^{k+1}, \vy^{k+1}, \vz^{k+1}) \right] \ge \sum_{k=0}^{\infty} \left(\frac{\rho}{2} - \frac{L^2}{\rho}\right) \|\vy^k - \vy^{k+1}\|_2^2,
$$
and the telescoping sum on the left-hand side can be simplified to
$$
\mathcal{L}(\vx^{0}, \vy^{0}, \vz^{0}) - \lim_{k \to \infty} \mathcal{L}(\vx^{k}, \vy^{k}, \vz^{k}) \ge \left(\frac{\rho}{2} - \frac{L^2}{\rho}\right) \sum_{k=0}^{\infty} \|\vy^k - \vy^{k+1}\|_2^2.
$$
Then the sequence $\mathcal{L}(\vx^{k}, \vy^{k}, \vz^{k})$ is bounded below implies that
$\sum_{k=0}^\infty \|\vy^k - \vy^{k+1}\|_2^2 < \infty$ and $\|\vy^k - \vy^{k+1}\|_2^2\to 0$ as $k\to \infty$, i.e., $\vy^{k+1}-\vy^k \to 0$. In addition, we can prove $\lim_{k\to \infty}\|\vz^{k+1}-\vz^k\|_2=0$ from \eqref{eqn:ADMMsc_gradientLip} and $\lim_{k\to \infty}\|\vx^{k+1}-\vx^k\|_2=0$ from \eqref{eqn:ADMMsc_c}.

From the optimality condition of \eqref{eqn:ADMMsc_a}, we know there exists $\vv\in \partial \lambda h_p(\vx^{k+1})$ such that
$$
0=\vv+\rho(\vx^{k+1}-\vy^k + \frac{\vz^k}{\rho}).
$$
By letting $\vu = \vv + \rho(\vx^{k+1}-\vy^{k+1} + \frac{\vz^{k+1}}{\rho})$, we have $\vu\in \partial_{\vx}\mathcal{L}(\vx^{k+1}, \vy^{k+1}, \vz^{k+1})$ and
\begin{equation}
    \|\vu\|_2 = \|\vz^{k+1}-\vz^k +\rho( \vy^k - \vy^{k+1})\|_2 \le \|\vz^{k+1}- \vz^k\|_2 + \rho\|(\vy^{k+1}- \vy^k)\|_2\le (L+\rho)\|(\vy^{k+1} -\vy^k)\|_2.
\end{equation}
Also, from the optimality condition of \eqref{eqn:ADMMsc_b} and \eqref{eqn:ADMMsc_c}, we have
\begin{equation}
    \|\nabla_{\vy}\mathcal{L}(\vx^{k+1}, \vy^{k+1}, \vz^{k+1})\|=\|\nabla\Phi(\vy^{k+1})+\rho(\vy^{k+1}-\vx^{k+1})-\vz^{k+1}\|=\|\vz^{k+1}-\vz^k\|_2\le L\|\vy^{k+1}-\vy^k\|_2
\end{equation}
and
\begin{equation}
\rho\|\vx^{k+1}-\vy^{k+1}\|_2 = \|\vz^{k+1}-\vz^k\|_2 \le  L\|\vy^{k+1}-\vy^k\|_2.
\end{equation}
Since $\lim_{k\to \infty}\|\vy^{k+1}-\vy^k\|=0$, the accumulation point $(\vx^\star, \vy^\star, \vz^\star)$ is a stationary point of \eqref{eqn:Lagrangian_CS}. Now $\vx^\star=\vy^\star$ and $\vz^\star=\nabla\Phi(\vy^\star)$, then we have $0\in \partial\lambda h_p(\vx^\star)+\nabla\Phi(\vx^\star)$, which completes the proof.
\end{proof}

\begin{remark}
    Different from coercive functions like $\ell_p$ and $\ell_1 - \ell_2$, the function $h_p$ is not coercive. Therefore we have to assume the sequence is bounded to prove the convergence.
\end{remark}

\subsection{Numerical Experiments: Sparse Signal Recovery}
This section evaluates the performance of the proposed algorithms H2-FBS, H2-APG and H2-ADMM in sparse signal recovery. Numerical experiments are conducted using synthetic data to demonstrate the effectiveness and robustness of our methods in different scenarios. All experiments were implemented in MATLAB 9.8 (R2020a) on a desktop with an Intel i7-7700 CPU (4.20GHz). 

The experiments employ two types of sensing matrices:

\smallskip
\noindent \emph{Oversampled discrete cosine transform (DCT) matrix:}  The $m \times n$ sensing matrix $\mA$ is constructed such that its $j$th column is defined as:  
        \begin{equation*}
            \va_j := \frac{1}{\sqrt{m}}\cos\bigg(\frac{2\pi j \vw}{E}\bigg),
        \end{equation*}
where $\vw$ is a random vector uniformly distributed over $[0,1]^m$, and $E>0$ is a parameter that controls the coherence of the matrix. Larger values of $E$ result in greater coherence, making the recovery problem more challenging. The oversampled DCT matrix is widely used in various applications, particularly in scenarios where traditional $\ell_1$-minimization models face difficulties due to high coherence \cite{fannjiang2012coherence,Rahimi-Wang-Dong-Lou:SIAMSC:2019,Yin-Lou-He-Xin:SIAMSC:2015}. 

\smallskip
\noindent \emph{Gaussian matrix:} 
The sensing matrix $\mA$ is generated from a multivariate normal distribution $\mathcal{N}(\mathbf{0}, \Sigma)$. For a given correlation coefficient $r\in [0,1]$, the covariance matrix $\Sigma$ is defined as:
$$
\Sigma_{ij}=\left\{
  \begin{array}{ll}
    1, & \hbox{if $i=j$;} \\
    r, & \hbox{otherwise.}
  \end{array}
\right.
$$
Here, a larger $r$ value indicates a more challenging problem in sparse recovery \cite{zhang2018minimization,Rahimi-Wang-Dong-Lou:SIAMSC:2019}. This setup is commonly used to assess the performance of recovery algorithms under varying levels of correlation.

The ground truth $\vx\in \mathbb{R}^n$ is simulated as an $s$-sparse signal, where $s$ is the number of non-zero entries in $\vx$. We require the indexes of the non-zero entries to be separated at least $2E$. The values of non-zero entries are set differently according to the sensing matrix $\mA$. For $\mA$ being an oversampled DCT matrix,   the dynamic range of a signal $\vx$ is defined as 
$$
\Theta(\vx)=\frac{\max\{|x_i|: i \in \mathrm{supp} (\vx)\}}{\min\{|x_i|: i \in \mathrm{supp} (\vx)\}},
$$
which can be controlled by an exponential factor $D$. Particularly, a Matlab command for generating those non-zero entries is 
$$\texttt{xg=sign(randn(s,1)).*10.\^{} (D*rand(s,1))},
$$ 
as used in  \cite{Wang-Yan-Rahimi-Lou:IEEESP:2020}. For example, $D=1$ and $D=5$ correspond to $\Theta \approx 10^1$ and $\Theta \approx 10^5$, respectively. For $\mA$ being a Gaussian random matrix, all $s$ non-zero entries of the sparse signal follow the Gaussian distribution $\mathcal{N}(0,1)$. 

In our experiments, the size of sensing matrices $\mA^{m \times n}$ is consistently set to  $64 \times 1024$. Due to the nonconvex nature of the model, the initial guess $\vx^{(0)}$ for an algorithm significantly influences the final result.  We choose $\vx^{(0)}$ as the $\ell_1$ solution obtained by Gurubi for all testing algorithms. All algorithms terminate when the  relative error between $\vx^{k}$ and $\vx^{k-1}$ is smaller than $10^{-6}$ or iteration number greater than $5n$. 

\subsubsection{Algorithmic Comparison}\label{experiment:H2comparisons}
In this section, we evaluate our proposed algorithms H2-FBS, H2-APG, and H2-ADMM by addressing the noise-free sparse signal recovery problem using oversampled DCT matrices. We consider a sparsity level of  $s = 5$  and coherence levels  $E \in \{1, 10\}$. The dynamic range is set to  $D \in \{3, 5\}$, parameters under which our prior research \cite{jia2025sparse} has demonstrated that models like  $h_2$, which are scale and signed permutation invariant, exhibit superior performance. In this context, we set the regularization parameter  $\lambda = 0.0001$  for problem \eqref{model:uncons}. For H2-FBS and H2-APG, the step size  $\gamma$  is set to  $0.99 / \|\mA\|_2^2$; for H2-ADMM, the penalty parameter  $\rho$  is set to 1.

Figure \ref{fig:Objvalue-comp-DCT} presents the objective function values of problem \eqref{model:uncons} plotted against the iteration number $k\in \N$ for the algorithms H2-FBS, H2-APG, and H2-ADMM. All methods exhibit a decrease in objective value as $k$ increases. We noticed that the objective function values for H2-APG and H2-ADMM decreased faster than those for H2-FBS. 
%%%%%%%%%%%%%%%%%%%%%%%%%%%%%
\begin{figure}[htb]

\centering
\begin{tabular}{cc}
\includegraphics[width=2.2in]{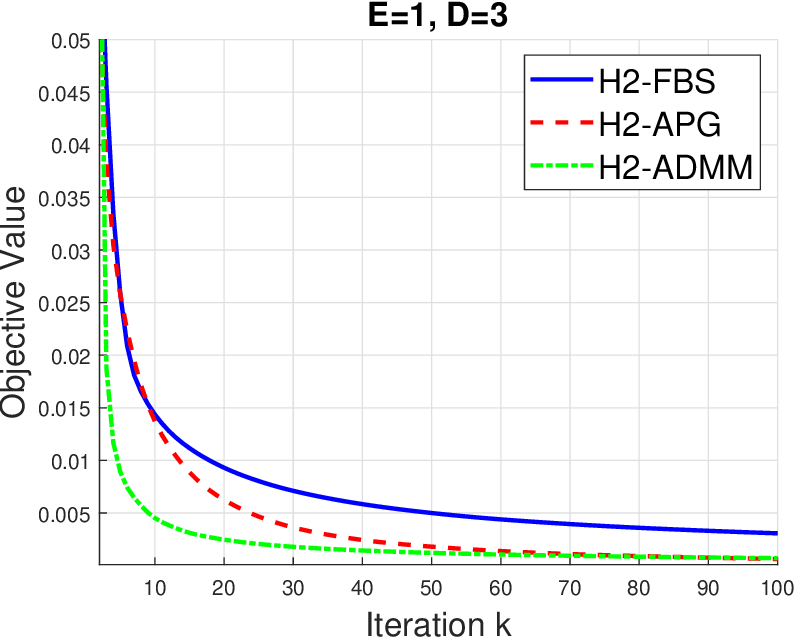}&
\includegraphics[width=2.2in]{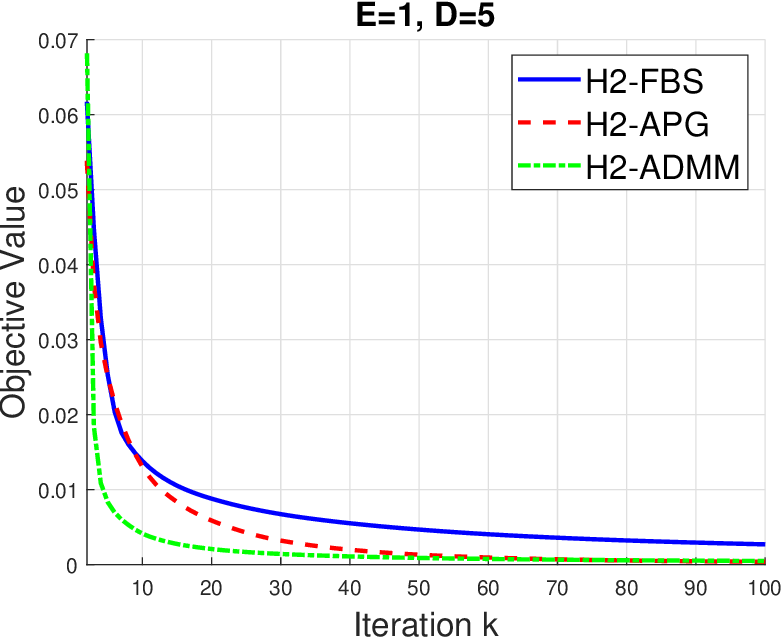}\\
\footnotesize{(a)}&\footnotesize{(b)}\\[0.1in]
\includegraphics[width=2.2in]{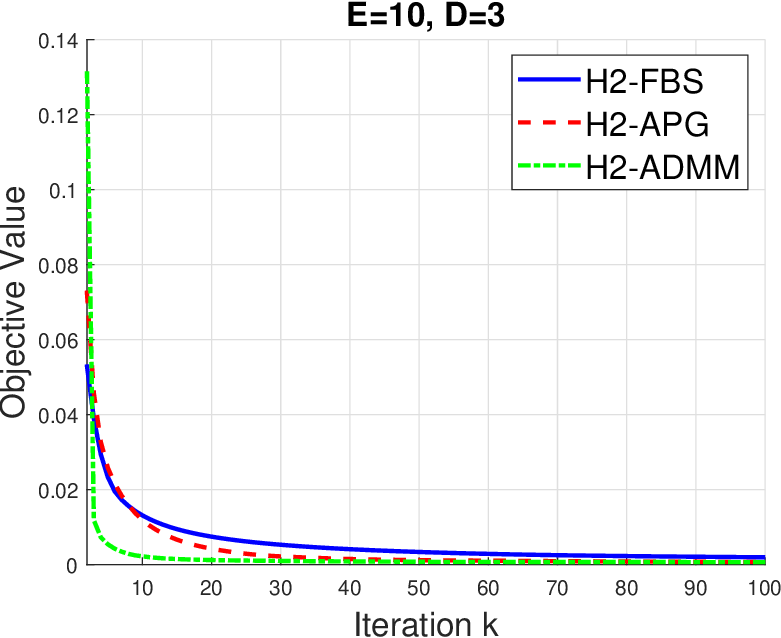}&
\includegraphics[width=2.2in]{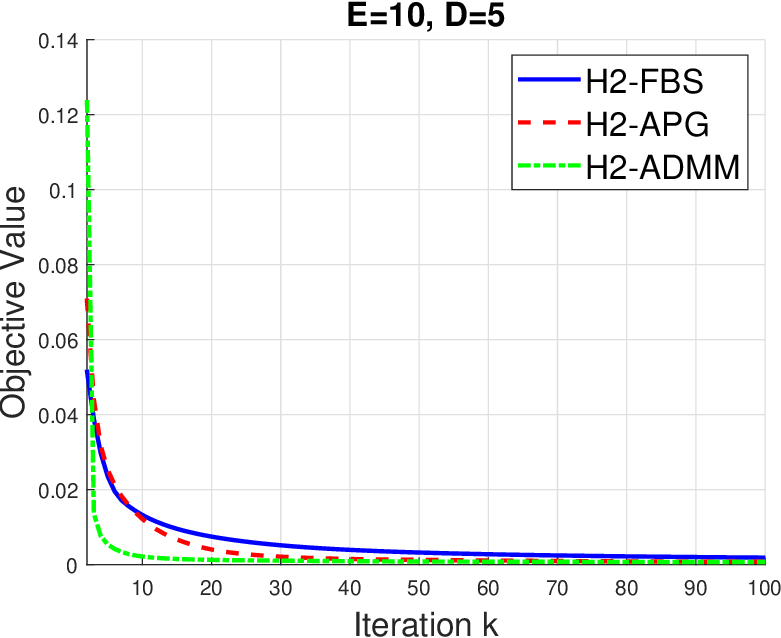}\\
\footnotesize{(c)}&\footnotesize{(d)}
\end{tabular}
\caption{Algorithmic comparison of objective values for problem \eqref{model:uncons} using DCT sensing matrices, presented in a \(2 \times 2\) grid format. Rows correspond to \(E=1,10\) from top to bottom, and columns correspond to \(D=3,5\) from left to right.}
\label{fig:Objvalue-comp-DCT}
\end{figure}

We further illustrate the discrepancy between the ground truth vector $\vx$ and the reconstructed solution $\vx^*$,  represented by the relative error $\|\vx^* - \vx\|_2/\|\vx\|_2$, plotted against the number of iterations in Figure~\ref{fig:errvalue-comp-DCT}. The results clearly indicate that the sequences generated by H2-APG and H2-ADMM converge significantly faster than those by H2-FBS. In scenarios involving low-coherence sensing matrices ($E=1$), H2-APG and H2-ADMM perform comparably. However, H2-APG demonstrates slightly superior computational efficiency in high-coherence scenarios ($E=10$). Additionally, as the complexity of the sparse signal recovery problem increases with higher $E$ and $D$ values, all compared algorithms require more iterations to converge.

Notably, H2-APG utilizes a single hyper-parameter $\gamma = \frac{0.99}{\|\mA\|_2^2}$, which proves to be robust and straightforward to implement, facilitating parameter tuning and adaptation to various scenarios. Overall, we rate H2-APG as the most effective algorithm for minimizing the model \eqref{model:uncons}, considering both recovery accuracy and computational efficiency.

\begin{figure}[htb]
\centering
\begin{tabular}{cc}
\includegraphics[width=2.2in]{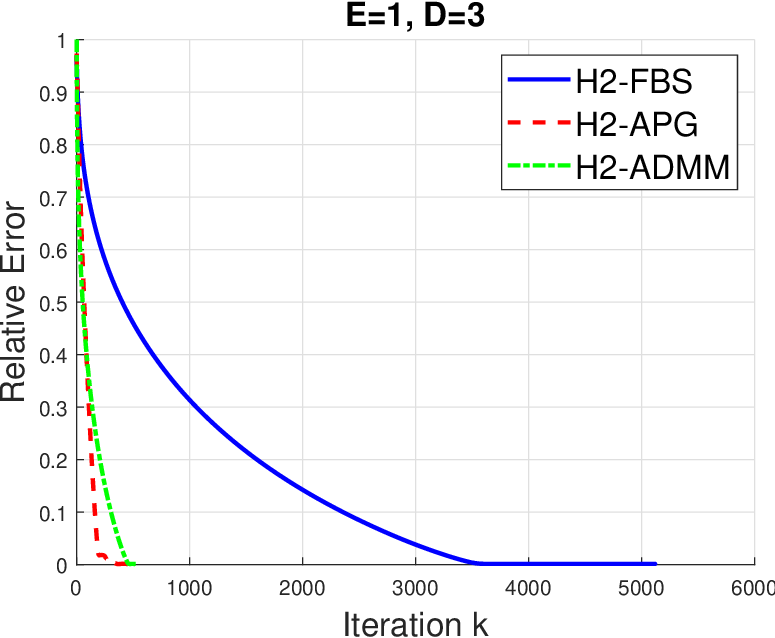}&
\includegraphics[width=2.2in]{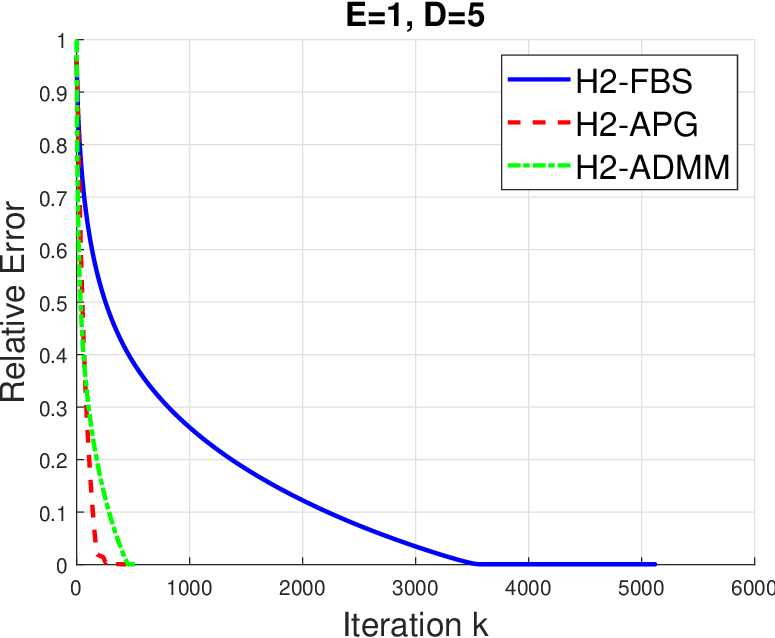}\\
\footnotesize{(a)}&\footnotesize{(b)}\\[0.1in]
\includegraphics[width=2.2in]{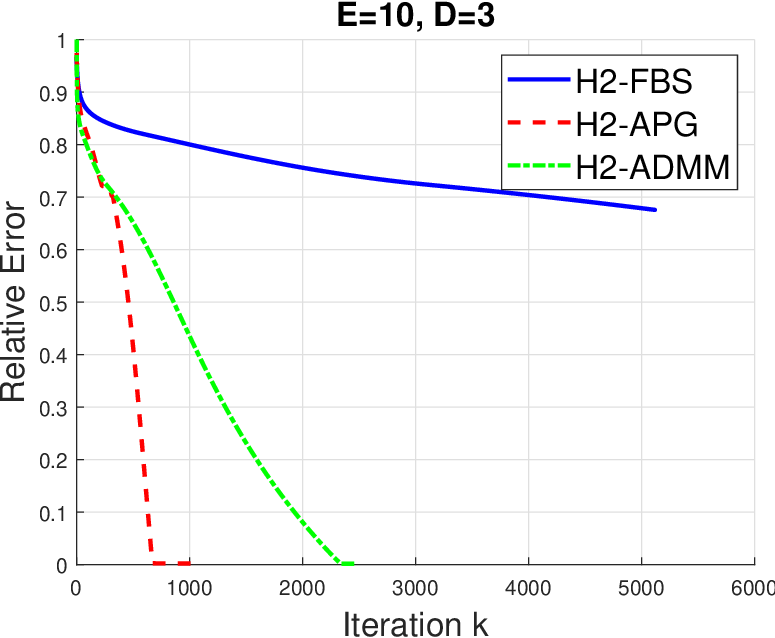}&
\includegraphics[width=2.2in]{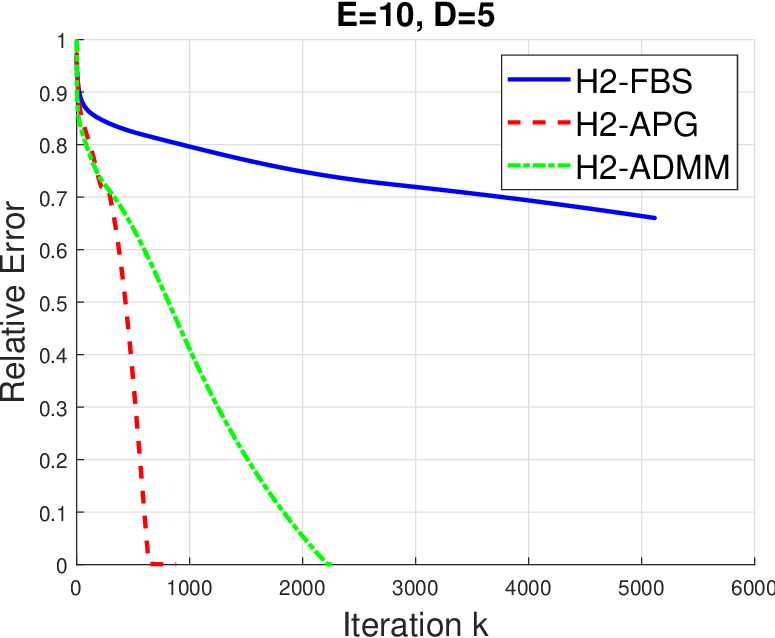}\\
\footnotesize{(c)}&\footnotesize{(d)}
\end{tabular}
\caption{Algorithmic comparison of relative error for problem \eqref{model:uncons} with DCT sensing matrices, presented in a \(2 \times 2\) grid format. Rows correspond to \(E=1, 10\) from top to bottom, and columns correspond to \(D=3, 5\) from left to right.}
 \label{fig:errvalue-comp-DCT}
\end{figure}

\subsubsection{Noise-free Measurements}

In this section, we focus on the performance of sparse signal recovery algorithms in the noise-free case for oversampled DCT matrices. Since the efficiency of the $\ell_1/\ell_2$ and $(\ell_1/\ell_2)^2$ based models compared with other sparse promoting models have been extensively studied in \cite{Wang-Yan-Rahimi-Lou:IEEESP:2020,jia2025sparse}, we only compare the signal recovery capacity of the proposed algorithms with the state-of-the-art algorithms for solving the $\ell_1/\ell_2$ and $(\ell_1/\ell_2)^2$ minimization problems,
\begin{equation}\label{model:constrained}
    \min\left\{h_p(\vx): \mA\vx= \vb, \vx\in \R^n\right\},
\end{equation}
where $p=1$ for $\ell_1/\ell_2$ model and $p=2$ for $(\ell_1/\ell_2)^2$ model. More detailed, we compare our algorithms H2-APG, H2-ADMM with the $\ell_1/\ell_2$-A1 algorithm in \cite{Wang-Yan-Rahimi-Lou:IEEESP:2020} and $(\ell_1/\ell_2)^2$ -LPMM, $(\ell_1/\ell_2)^2$ -QP algorithms in \cite{jia2025sparse}, as they are rated as the most efficient algorithms in those works. For simplicity, we denote the three algorithms $\ell_1/\ell_2$-A1, $(\ell_1/\ell_2)^2$-LPMM, $(\ell_1/\ell_2)^2$-QP as L-A1, D-LPMM, D-QP respectively.

We evaluate the performance of sparse signal recovery in terms of three metrics: success rate, the number of missing nonzero coefficients, and the number of misidentified nonzero coefficients. More precisely, the success rate is the number of successful trials over the total number of trials, and success is declared if the relative error of the reconstructed solution $\vx^*$ to the ground truth vector $\vx$ is no more than than $0.005$, i.e., $\|\vx^*-\vx\|_2/\|\vx\|_2\le 0.005$; the number of missing nonzero coefficients refers to the number of nonzero coefficients that an algorithm ``misses", i.e., determines to be zero; the number of misidentified nonzero coefficients refers to the number of nonzero coefficients that are ``misidentified", i.e., coefficients that are determined to be nonzero when they should be zero. The success rate is calculated over 50 trials, and the larger the success rate is, the better the reconstructed signal will be. The last two metrics measure how well each algorithm finds the signal support, meaning the location of the nonzero coefficients. These two metrics are calculated by the average of values over 50 trials. The smaller the values of these two metrics are, the better the reconstructed signals will be. 

In all settings of this subsection, we set $\lambda = 10^{-4}$ for the unconstrained model \eqref{model:uncons}. The other parameters of H2-APG and H2-ADMM are the same as those in subsection~\ref{experiment:H2comparisons}, and the parameter settings of L-A1 and D-LPMM are chosen as suggested in \cite{Wang-Yan-Rahimi-Lou:IEEESP:2020} and \cite{jia2025sparse}, where D-QP algorithm in \cite{jia2025sparse} have the advantage of parameter-free.

Figure~\ref{fig:success-comp-DCT} displays the success rate across 50 trials for various algorithms as the sparsity parameter $s$ ranges from $\{2, 6, 10, 14, 18, 22\}$. We consider different coherence matrices and dynamic ranges of signals, where $F\in\{1, 10, 20\}$ and $D\in\{0, 1, 5\}$. As shown in Figure~\ref{fig:success-comp-DCT}, our proposed algorithms perform comparably to L-A1, D-LPMM, and D-QP, and in some cases of lower sparsity, even outperform these benchmarks. For signals with high dynamic range and larger sparsity values, the performance of our proposed algorithms could be further enhanced. Table~\ref{tab:accuracy-comp-DCT} reports the accuracy of the tested algorithms, measured by the number of missing and misidentified nonzero coefficients. All algorithms exhibit similar performance on these metrics. Despite the suboptimal performance of our proposed algorithms in cases of high dynamic range and larger sparsity, their ability to accurately identify the location of nonzero entries suggests that a post-processing step, such as applying a least squares method, could effectively recover the signal.

% %%%%%%%%%%%%%%%%%%%%%%%%%%%%%%%%%%%%%%%%%%%%%%%%%%%%
\begin{figure}[htb]
    \centering
    \begin{subfigure}{0.3\textwidth}
        \includegraphics[width=\linewidth]{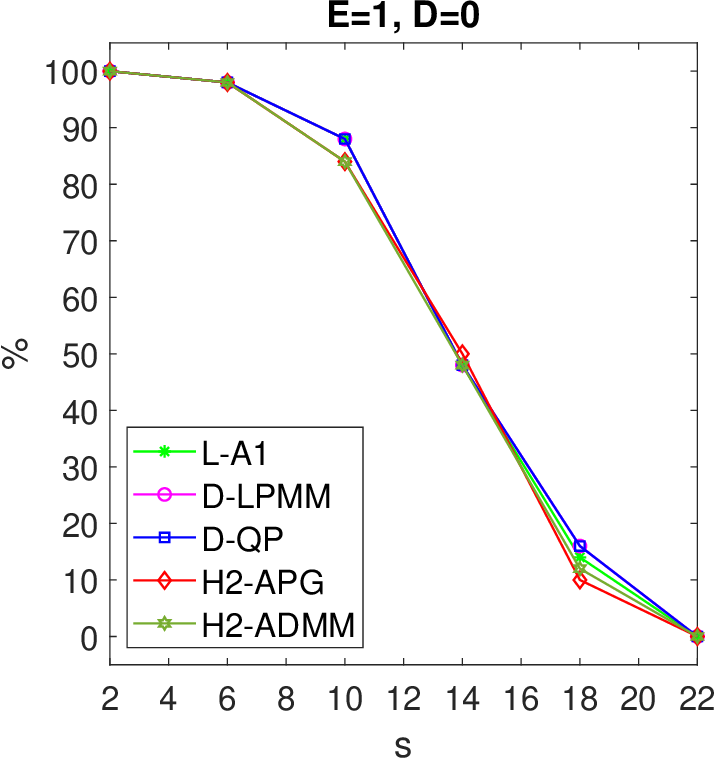}
    \end{subfigure}
    \hfill
    \begin{subfigure}{0.3\textwidth}
        \includegraphics[width=\linewidth]{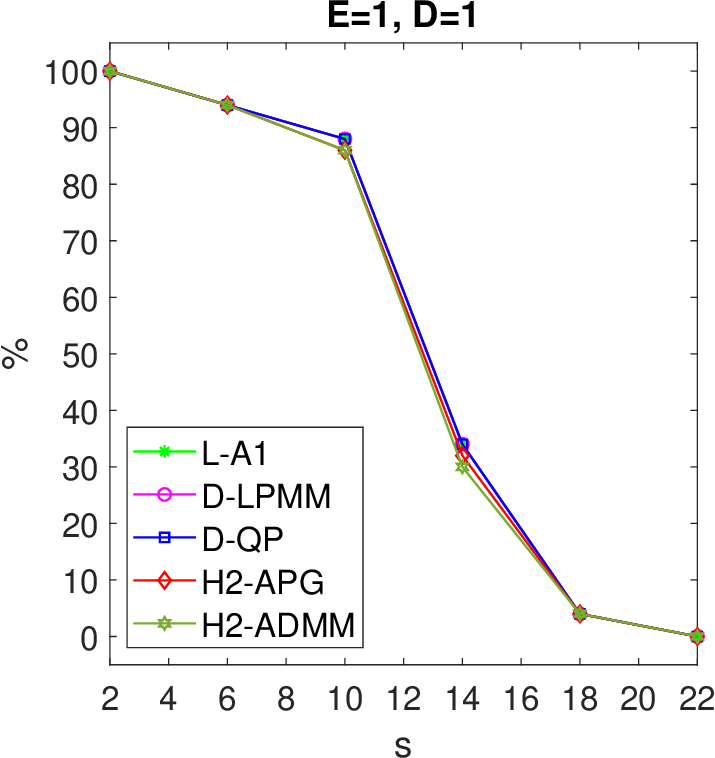}
    \end{subfigure}
    \hfill
    \begin{subfigure}{0.3\textwidth}
        \includegraphics[width=\linewidth]{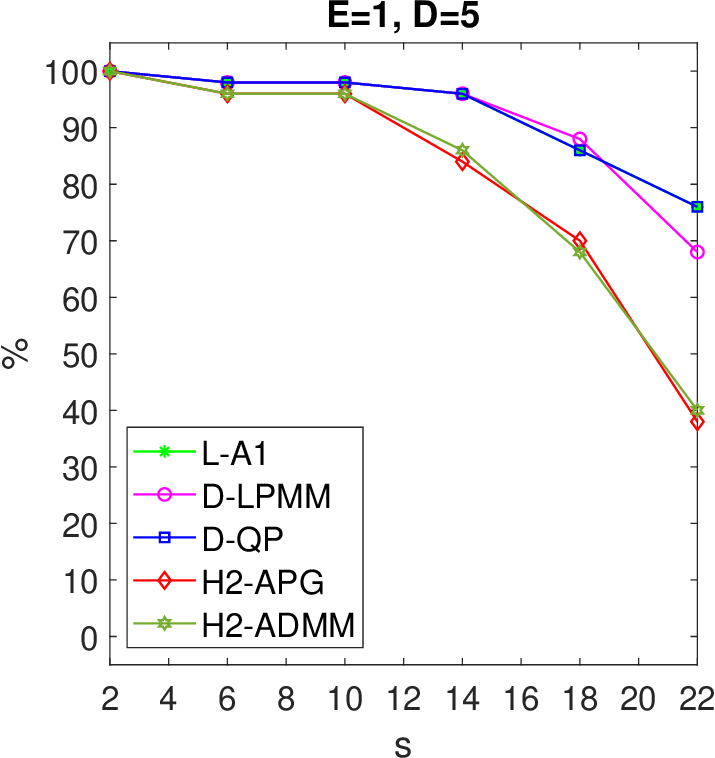}
    \end{subfigure}

    \medskip

    \begin{subfigure}{0.3\textwidth}
        \includegraphics[width=\linewidth]{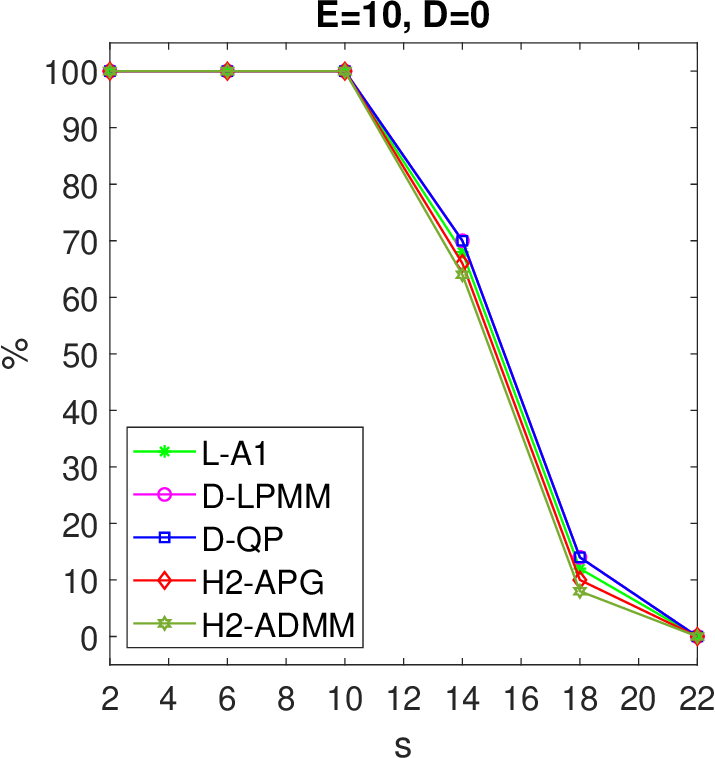}
    \end{subfigure}
    \hfill
    \begin{subfigure}{0.3\textwidth}
        \includegraphics[width=\linewidth]{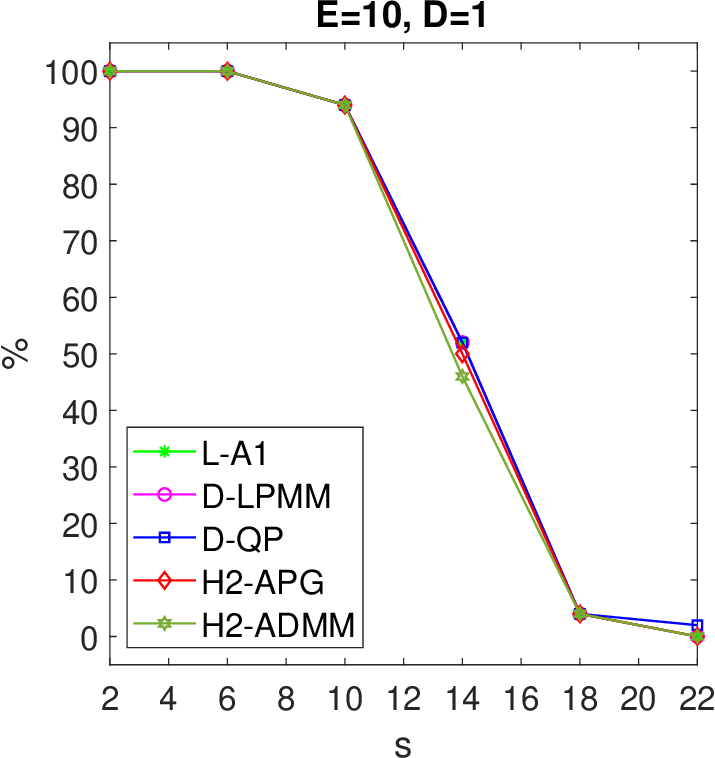}
    \end{subfigure}
    \hfill
    \begin{subfigure}{0.3\textwidth}
        \includegraphics[width=\linewidth]{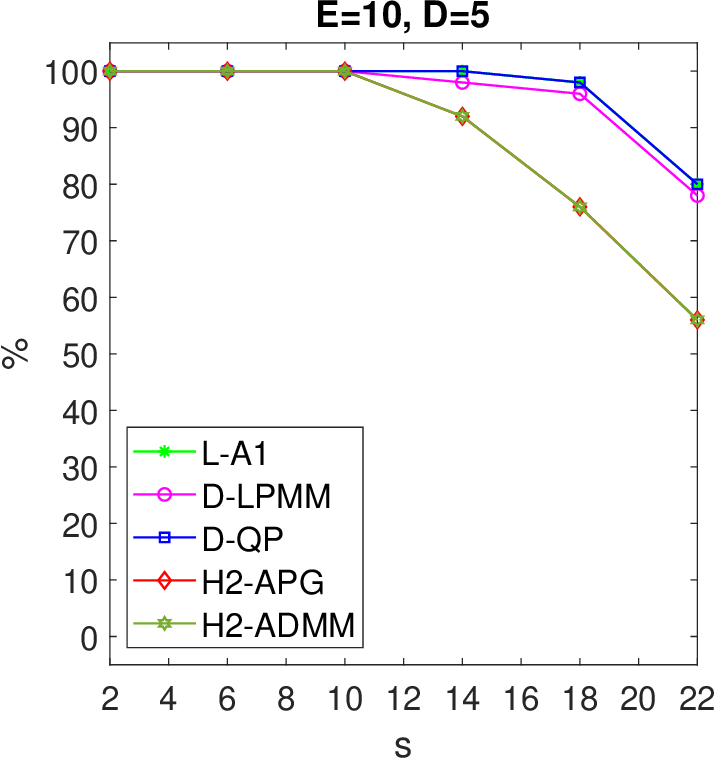}
    \end{subfigure}

    \medskip

    \begin{subfigure}{0.3\textwidth}
        \includegraphics[width=\linewidth]{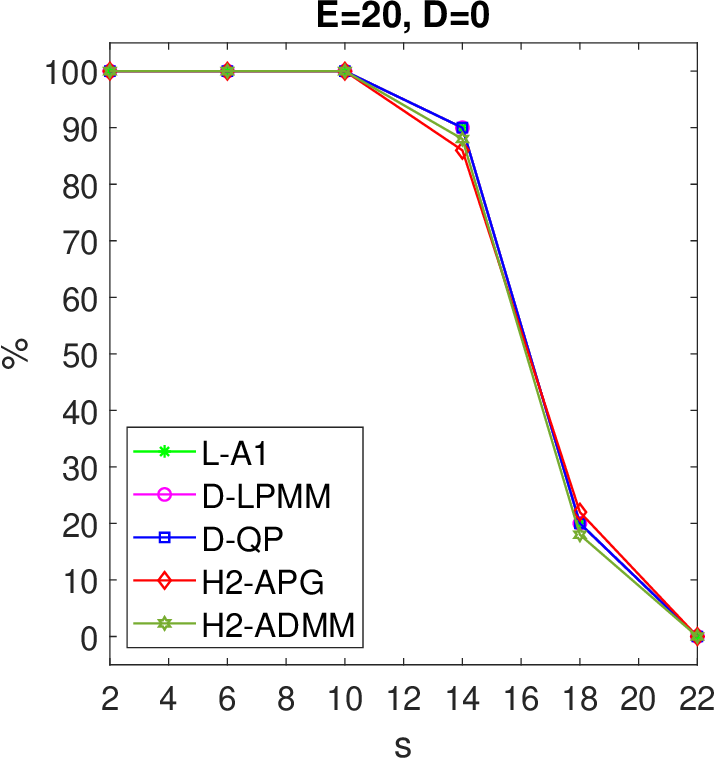}
    \end{subfigure}
    \hfill
    \begin{subfigure}{0.3\textwidth}
        \includegraphics[width=\linewidth]{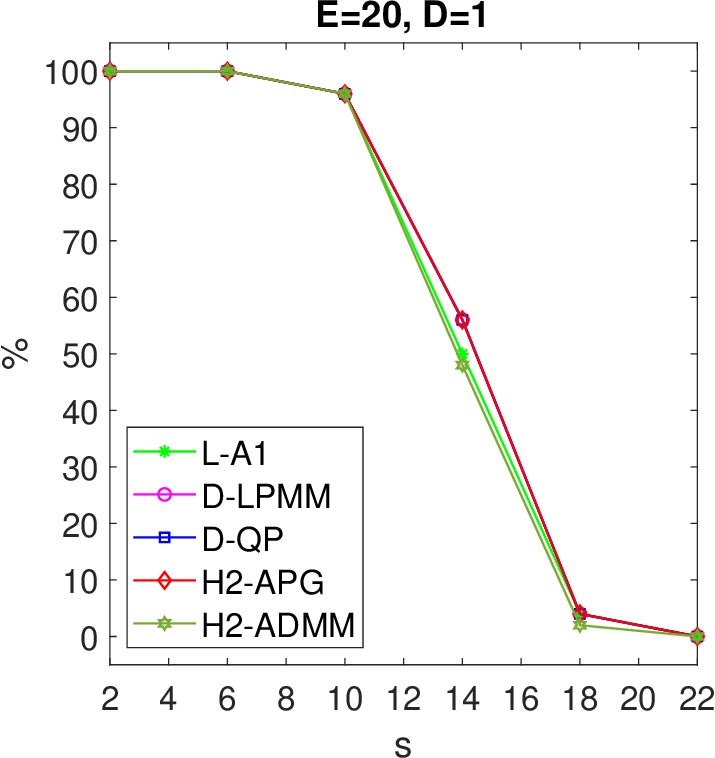}
    \end{subfigure}
    \hfill
    \begin{subfigure}{0.3\textwidth}
        \includegraphics[width=\linewidth]{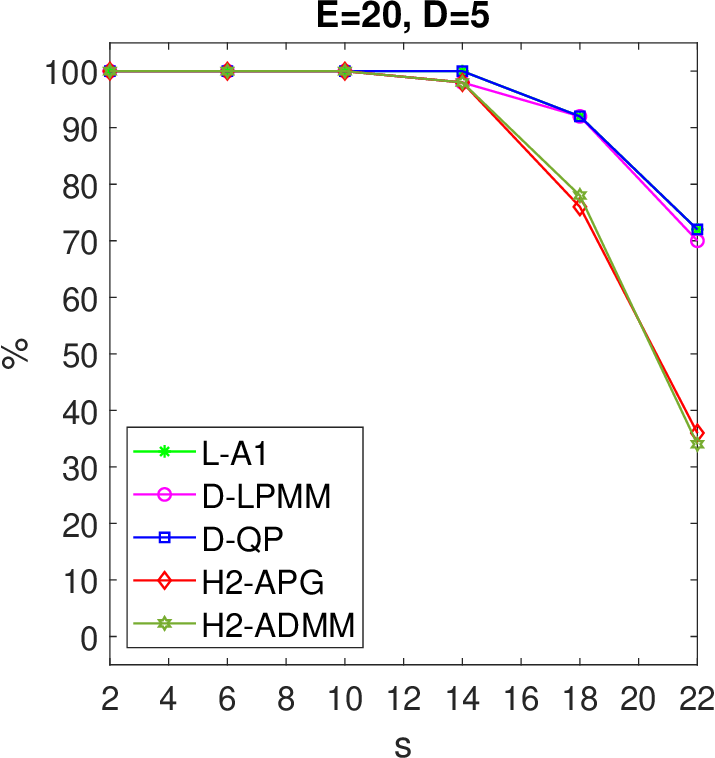}
    \end{subfigure}

    \caption{Algorithmic comparison of success rates is conducted for  DCT sensing matrices. The presentation of results is structured in a $3 \times 3$ grid format, where rows correspond to $E=1,10, 20$ from top to bottom, and columns correspond to $D=0,1,5$ from left to right.}

    \label{fig:success-comp-DCT}
\end{figure}

\begin{table}[htbp!] 
\caption{Accuracy of algorithms for solving problem \eqref{model:uncons} in noise-free setting }
\centering
\resizebox{\textwidth}{!}
{
\begin{tabular}{cccccc|ccccc}
\hline
\noalign{\hrule height 1pt}
&\multicolumn{5}{c|}{\# missing nonzero coeffs} &\multicolumn{5}{c}{\# misidentified nonzero coeffs}\\ 
\hline
s       & L-A1 & D-LPMM & D-QP & H2-APG & H2-ADMM & L-A1 & D-LPMM & D-QP & H2-APG & H2-ADMM\\ 
\noalign{\hrule height 1pt}
&\multicolumn{10}{c}{Oversampled DCT with $(E, D)= (10, 3)$}\\ 
\hline
2       & 0.00 & 0.00 & 0.00 & 0.00 & 0.00 & 0.00 & 0.00 & 0.00 & 0.00 & 0.00 \\
6       & 0.00 & 0.00 & 0.00 & 0.00 & 0.00 & 0.00 & 0.00 & 0.00 & 0.00 & 0.00 \\
10      & 0.00 & 0.00 & 0.00 & 0.00 & 0.00 & 0.00 & 0.00 & 0.00 & 0.00 & 0.00 \\
14      & 0.28 & 0.32 & 0.28 & 0.28 & 0.26 & 4.96 & 5.40 & 4.96 & 4.14 & 4.76 \\
18      & 1.36 & 1.84 & 1.36 & 1.10 & 1.04 & 17.54 & 22.04 & 17.60 & 13.72 & 15.24 \\
22      & 6.38 & 6.24 & 6.24 & 5.02 & 5.06 & 41.32 & 38.24 & 40.00 & 30.04 & 31.92 \\
\noalign{\hrule height 0.5pt}
&\multicolumn{10}{c}{Oversampled DCT with $(E, D)= (20, 3)$}\\ 
\hline
2       & 0.00 & 0.00 & 0.00 & 0.00 & 0.00 & 0.00 & 0.00 & 0.00 & 0.00 & 0.00 \\
6       & 0.00 & 0.00 & 0.00 & 0.00 & 0.00 & 0.00 & 0.00 & 0.00 & 0.00 & 0.00 \\
10      & 0.00 & 0.00 & 0.00 & 0.00 & 0.00 & 0.00 & 0.00 & 0.00 & 0.00 & 0.00 \\
14      & 0.28 & 0.50 & 0.28 & 0.18 & 0.18 & 4.90 & 6.96 & 4.90 & 3.78 & 3.36 \\
18      & 2.20 & 2.52 & 2.24 & 1.94 & 2.08 & 23.48 & 25.06 & 23.56 & 21.96 & 20.98 \\
22      & 7.84 & 8.20 & 8.02 & 4.38 & 4.28 & 45.64 & 43.38 & 45.84 & 40.38 & 39.26 \\
\noalign{\hrule height 1pt}
\end{tabular}
}
\label{tab:accuracy-comp-DCT}
\end{table}

\subsubsection{Noisy Data}

In this subsection, we consider the sparse signal recovery problem with white Gaussian noise. In order to test the performance of the proposed algorithms in solving $\ell_1/\ell_2$ and $(\ell_1/\ell_2)^2$ based sparse recovery problems, we first present various computational aspects of the proposed algorithms together with comparison to MBA in \cite{Zeng-Yi-Pong:SIAMOP:2021}, then we compare various sparse promoting models.

\begin{table}[hb!] 
\caption{Relative error for Gaussian noise signal recovery with oversampled DCT matrices}
\centering

\begin{tabular}{ccc|cccc}
\hline
\noalign{\hrule height 1pt}
\multicolumn{3}{c|}{settings} &\multicolumn{4}{c}{Relative error}\\ 
\hline
s       & E & D & spg$\ell_1$ & MBA & H2-APG & H2-ADMM \\ 
\noalign{\hrule height 1pt}
\hline
4 & 1 & 2 & 3.391e-03 & 2.053e-03 & 1.492e-03 & \textbf{1.491e-03} \\
4 & 1 & 3 & 3.227e-03 & 2.117e-03 & \textbf{1.458e-03} & 1.459e-03 \\
4 & 5 & 2 & 4.188e-03 & 2.907e-03 & \textbf{2.190e-03} & 2.192e-03 \\
4 & 5 & 3 & 1.081e-03 & 9.116e-04 & 6.607e-04 & \textbf{6.577e-04} \\
4 & 15 & 2 & 5.463e-01 & \textbf{2.134e-01} & 2.157e-01 & 2.867e-01 \\
4 & 15 & 3 & 2.490e-01 & \textbf{6.024e-02} & 7.340e-02 & 1.100e-01 \\
\noalign{\hrule height 0.5pt}
\hline
8 & 1 & 2 & 2.974e-03 & 2.227e-03 & \textbf{2.055e-03} & \textbf{2.055e-03} \\
8 & 1 & 3 & 5.583e-03 & 2.421e-04 & \textbf{2.211e-04} & 4.723e-03 \\
8 & 5 & 2 & 7.617e-03 & 3.633e-03 & \textbf{3.546e-03} & 3.547e-03 \\
8 & 5 & 3 & 1.563e-03 & 4.064e-04 & \textbf{3.864e-04} & 6.775e-04 \\
8 & 15 & 2 & 5.469e-01 & \textbf{1.034e-01} & 2.573e-01 & 1.984e-01 \\
8 & 15 & 3 & 6.491e-01 & 6.337e-01 & \textbf{5.300e-01} & 6.857e-01 \\
\noalign{\hrule height 0.5pt}
\hline
12 & 1 & 2 & 9.376e-02 & 1.686e-03 & \textbf{1.494e-03} & 4.714e-03 \\
12 & 1 & 3 & 2.027e-02 & 3.763e-04 & \textbf{3.644e-04} & 1.626e-02 \\
12 & 5 & 2 & 1.244e-01 & 7.536e-03 & \textbf{3.658e-03} & 1.124e-02 \\
12 & 5 & 3 & 3.572e-02 & \textbf{1.323e-02} & 1.407e-02 & 3.308e-02 \\
12 & 15 & 2 & 6.153e-01 & \textbf{3.506e-01} & 6.326e-01 & 3.785e-01 \\
12 & 15 & 3 & 5.750e-01 & 5.862e-01 & \textbf{3.775e-01} & 6.033e-01 \\
\noalign{\hrule height 1pt}
\end{tabular}
\label{tab:relerr-comp-noisy}
\end{table}

We compare our proposed algorithms with MBA in \cite{Zeng-Yi-Pong:SIAMOP:2021} for sparse signal recovery with white Gaussian noise. The test instances and experiment settings are similar to \cite{Zeng-Yi-Pong:SIAMOP:2021}. Specifically, the sensing matrix $\mA$ and ground truth signal $\vx$ are same as above, the observation $\vb$ is set to be $\vb=\mA\vx + 0.01 \xi$, where $\xi\in \R^m$ is a random vector with i.i.d. standard Gaussian entries, and we set $\sigma = 1.2 \|0.01\xi\|_2$.  The initial points of the compared algorithms in this subsection are chosen the same as in \cite{Zeng-Yi-Pong:SIAMOP:2021}, which are
\begin{eqnarray*}
    \vx^0 = \begin{cases}
        \mA^\dag\vb + \sigma \frac{\vx_{\ell_1}-\mA^\dag \vb}{\|\mA \vx_{\ell_1}-\vb\|_2}, \quad\quad & \text{if }\|\mA\vx_{\ell_1}-\vb\|_2 > \sigma \\
        \vx_{\ell_1} \quad & \text{otherwise,}
    \end{cases}
\end{eqnarray*}
where $\vx_{\ell_1}$ is an approximate solution solved via SPGL1 \cite{BergFriedlander:SIAMSC:2008} (version 2.1) using default settings, and $\mA^\dag$ is the pseudo-inverse of $\mA$. Note that such $\vx^0$ is in the constraint of problem \eqref{model:constrained}. In the experiment, $s\in \{4, 8, 12\}$, $F\in \{1,5, 15\}$ and $D\in \{2, 3\}$. The parameter for the MBA method is set as in \cite{Zeng-Yi-Pong:SIAMOP:2021}. For our proposed algorithms, $\lambda$ is set to be $0.1$ and $0.01$ for H2-APG and H2-ADMM respectively, and other parameters are the same as those in Section~\ref{experiment:H2comparisons}.

Table~\ref{tab:relerr-comp-noisy} present the relative error of the reconstructed solution $\vx^*$ to the ground truth vector $\vx$ averaged over 20 trials, i.e., $\|\vx^*-\vx\|_2/\|\vx\|_2$. It can be observed from Table~\ref{tab:relerr-comp-noisy} that compared with MBA, our proposed method H2-APG obtains recovered signals with smaller recovery error in most cases.

Now we compare different sparse promoting models, including $\ell_1$, $\ell_1 - \ell_2$ and $(\ell_1/\ell_2)^2$ for Gaussian noise sparse recovery with Gaussian matrices. Here, the $\ell_1$, $\ell_1-\ell_2$ models are the least squares problems with the corresponding sparse promoting regularizers.

The experiment setup follows \cite{Lou-Yan:JSC2018}. The standard deviation of the white Gaussian noise is set to be $\sigma=0.1$. Specifically, we consider a signal $\vx$ of length $n=512$ and $s=100$ nonzero elements. The number of measurements $m$ varies from $230$ to $300$.  The sensing matrix $\mA$ is generated by normalizing each column of a Gaussian matrix to be zero-mean and unit norm.  We use the mean-square-error (MSE) to quantify the recovery performance, i.e., $\text{MSE}= \|\vx^\star - \vx\|_2$. Note that with the known support of the ground truth signal, denoted as $I=\text{supp}(\vx)$, then one can compute $\sigma^2 \text{tr}(\mA_I^\top \mA_I)^{-1}$ as a benchmark, where $\mA_I$ denotes the columns of matrix $\mA$ restricted to $I$.

In the experiments, the parameter for $\ell_1-\ell_2$ and $\ell_1$ models are set the same as in \cite{Lou-Yan:JSC2018}. For our proposed algorithms, H2-APG and H2-ADMM, we set $\lambda$ to be $1e-7$ and $1e-5$ respectively. All other settings are the same as in Section~\ref{experiment:H2comparisons}.

Tabel~\ref{tab:MSE-comp-noisy} presents the average of MSE over 20 trials of the compared algorithms with Gaussian matrices for four specific values of $m$: 238, 250, 276, and 300. These values were considered in previous studies \cite{Lou-Yan:JSC2018,Xu-Chang-Xu-Zhang:IEEENNLS:12}. It is observed that the proposed H2-APG algorithm consistently achieves lower MSE than the other compared models in most cases, particularly when the sensing matrix exhibits higher coherence (i.e. when $m$ is smaller).

\begin{table}[h!] 
\caption{MSE for Gaussian noise signal recovery with Gaussian matrices}
\centering

\begin{tabular}{c|ccccc}
\hline
\noalign{\hrule height 1pt}
\multicolumn{1}{c|}{settings} &\multicolumn{5}{c}{MSE}\\ 
\hline
m & Oracle & $\ell_1$ & $\ell_1 - \ell_2$ & H2-APG & H2-ADMM \\ 
\noalign{\hrule height 1pt}
\hline
238 & 3.0116 & 4.856 & 4.078 & \textbf{4.0493} & 4.0624\\
250 & 2.7518 & 4.0355 & 3.8095 & \textbf{3.7615} & 3.7738 \\
276 & 2.3531 & 3.3316 & 3.1305 & \textbf{3.0969} & 3.1031 \\
300 & 2.0592 & 2.9662 & \textbf{2.8037} & 2.8512 & 2.8150 \\
\noalign{\hrule height 1pt}
\end{tabular}
\label{tab:MSE-comp-noisy}
\end{table}

\section{Application: Image Restoration}
Image restoration addresses the challenge of recovering an unknown image, denoted as $\vy$, from a degraded image $\vb$ described by the relationship:
\begin{equation} \label{model:restoration}
    \vb = \mA\vy + \text{noise}
\end{equation}
with $\mA$ being some bounded linear operator. This operator, along with the introduced noise, naturally emerges during the imaging, acquisition, and communication processes, resulting in the corrupted observation $\vb$. The primary objective of the image
restoration is to approximate the true image $\vy$ from the observed image $\vb$, allowing for the extraction of subtle yet crucial features and objects within $\vy$. Consequently, image restoration holds significant importance in the broader field of image processing and analysis.

In general, the operator $\mA$ in Equation \eqref{model:restoration} is a singular and spatially varying matrix, rendering the problem of finding the solution $\vy$ ill-posed. Due to this inherent
ill-posedness, a typical variational method is often formulated as:
\begin{equation}
    \inf_{\vy\in \mathbb{U}}\{\mathcal{F}(\vy, \vb) + \lambda \cdot \mathcal{R}(\vy)\},
\end{equation}
where $\mathcal{F}(\vy, \vb)$ represents the data-fidelity term, $\mathcal{R}(\vy)$ is the regularization term, $\lambda$ is a
regularization parameter that balances the importance between the data-fidelity term and the regularization term, and $\mathbb{U}$ is a finite-dimensional vector space containing the unknown solution. This variational approach serves as a robust framework for
addressing the challenges posed by the ill-posed nature of the model \eqref{model:restoration} in image
restoration.

One typical model is given by 
\begin{equation}\label{model:ROF}
    \argmin_{\vy} \left\{\frac{1}{2}\|\mA\vy - \vb\|_2^2+\lambda\|\vy\|_{\text{TV}}\right\},
\end{equation}
where $\vy$, $\vb$ are the original and observed noisy image respectively, and $\|\vy\|_{\text{TV}}$ is the total variation of $\vy$.

Suppose a two-dimensional (2D) image is defined on an $m \times n$ Cartesian grid. By using a standard linear index, we can represent a 2D image as a vector, i.e. the $((i-1)\cdot m + j)$-th component denotes the intensity value at pixel $(i,j)$. Without loss of generality, we assume $n=m$ that the image is square.

Let $\mD$ denote the $n \times n$ first-order difference matrix defined by
\begin{equation*}
    \mD := \begin{bmatrix}
        0&&&\\
        -1&1&&\\
        &\ddots&\ddots&\\
        &&-1&1
    \end{bmatrix}
\end{equation*}
and choose the $\mB$ as a $2n^2 \times n^2$ matrix
\begin{equation}
    \mB := \begin{bmatrix}
        \mI_n \otimes \mD\\
        \mD \otimes \mI_n
    \end{bmatrix}=\begin{bmatrix}
        \mB_x\\\mB_y
    \end{bmatrix}, 
\end{equation}
where $\mI_n$ is the $n\times n$ identity matrix and the notation $\mP \otimes \mQ$ denotes the Kronecker product of matrices $\mP$ and $\mQ$. Here we note $\mB_x$ and $\mB_y$ are the forward difference matrix with Neumann Boundary (Type II) condition in the horizontal and vertical directions, respectively.

For an image $\vy\in \R^{n^2}$, the total variation of $\vy$ can be defined as 
\begin{equation}
    \|\vy\|_{\text{ATV}}:= \|\mB\vy\|_1=\sum_{i=1}^{n^2}\left|(\mB\vy)_i\right|+|(\mB\vy)_{n^2 + i}|=\sum_{i=1}^{n^2}\left\|\begin{bmatrix}
        (\mB\vy)_i\\(\mB\vy)_{n^2+i}
    \end{bmatrix}\right\|_1
\end{equation}
and
\begin{equation}\label{def:ITV}
    \|\vy\|_{\text{ITV}}:=\sum_{i=1}^{n^2}\left\|\begin{bmatrix}
        (\mB\vy)_i\\(\mB\vy)_{n^2+i}
    \end{bmatrix}\right\|_2= \sum_{i=1}^{n^2}\sqrt{(\mB\vy)_i^2+(\mB\vy)_{n^2 + i}^2}
\end{equation}
for anisotropic and isotropic cases, respectively.

Let $\mB_i$ be a two-row matrix formed by stacking the $i$-th row of $\mB_x$ and $\mB_y$. Then the anisotropic and isotropic TV restoration models can be written as
\begin{equation}
    \argmin_{\vy} \left\{\frac{1}{2}\|\mA\vy - \vb\|_2^2+\lambda\sum_{i=1}^{n^2}\|\mB_i\vy\|_1\right\}
\end{equation}
and 
\begin{equation}\label{model:ITV}
    \argmin_{\vy} \left\{\frac{1}{2}\|\mA\vy - \vb\|_2^2+\lambda\sum_{i=1}^{n^2}\|\mB_i\vy\|_2\right\}.
\end{equation}

\subsection{Algorithm: Image Restoration}
In this section, we examine the model
\begin{equation}\label{model:TVhp}
    \argmin_{\vy} \left\{\frac{1}{2}\|\mA\vy - \vb\|_2^2+\lambda\sum_{i=1}^{n^2}h_p\left(\mB_i\vy\right)\right\}\quad \text{for}\;\; p=1,2
\end{equation}
By defining $\varphi:\R^{2n^2}\to \R$ as
\begin{equation}\label{def:phi}
    \varphi(\vu):=\sum_{i=1}^{n^2} h_p\left(\begin{bmatrix}
        \vu_i \\ \vu_{n^2+i}
    \end{bmatrix}\right)
\end{equation}
then the model \eqref{model:TVhp} is equivalent to 
\begin{equation}\label{model:TVhp_rewrite}
        \argmin_{\vy} \left\{\frac{1}{2\lambda}\|\mA\vy - \vb\|_2^2+\varphi(\mB\vy)\right\}
\end{equation}
Note that follows from definition in \eqref{def:prox}, we have \eqref{model:TVhp_rewrite} as evaluating proximity operator $\prox_{\varphi\circ \mB}$ at $\vb$.

The ADMM is the most favorable method for solving optimization problems. To this end, we rewrite problem \eqref{model:TVhp_rewrite} as
\begin{align} \label{model:TVhp_ADMM}
    \min_{\vx\in \R^{2n^2}, \vy \in \R^{n^2}} &\quad \lambda \varphi(\vx) + \frac{1}{2}\|\mA\vy - \vb\|_2^2\\
    s.t. & \quad \vx = \mB\vy \nonumber
\end{align}
Then the corresponding augmented Lagrangian to \eqref{model:TVhp_ADMM} is
\begin{equation} \label{eqn:TVhpADMMLfun}
    \mathcal{L}_{\rho}(\vx, \vy, \vz) = \lambda \varphi(\vx) + \frac{1}{2}\|\mA\vy - \vb\|_2^2 + \la \vz, \vx - \mB\vy \ra + \frac{\rho}{2}\|\vx - \mB\vy\|_2^2,
\end{equation}
where $\vz$ is the Lagrangian multiplier (dual variable) and $\rho>0$ is a parameter. The ADMM scheme iterates as follows:
\begin{subequations} \label{eqn:ADMMsc_IR}
\begin{align}
        \vx^{k+1} &\in \argmin_{\vx\in \R^{2n^2}} \left(\lambda \varphi(\vx)+ \frac{\rho}{2} \left\|\vx - \mB\vy^k + \frac{\vz^k}{\rho}\right\|\right)\label{eqn:ADMMsc_x},\\
        \vy^{k+1} &= \argmin_{\vy\in \R^{n^2}} \left( \frac{1}{2}\|\mA\vy - \vb\|_2^2+ \frac{\rho}{2}\left\| \mB\vy - \vx^{k+1} - \frac{\vz^k}{\rho}\right\|_2^2 \right)\label{eqn:ADMMsc_y},\\ 
        \vz^{k+1} &= \vz^k + \rho (\vx^{k+1} - \mB\vy^{k+1}).   \label{eqn:ADMMsc_z} 
\end{align}
\end{subequations}

The $\vx$-subproblem is 
\begin{equation}\label{eqn:v-update}
    \vx^{k+1}\in \text{prox}_{\frac{\lambda}{\rho}\varphi}\left(\mB\vy^{k}-\frac{\vz^k}{\rho}\right).
\end{equation}
Since $\varphi$ defined in \eqref{def:phi} is separable, then $\prox_{\frac{\lambda}{\rho}\varphi}$ can be computed using \cite[Algorithm 1]{jia2024computing} and \cite[Algorithm 2]{jia2024computing}.

The $\vy$-subproblem has a closed-form solution given by
\begin{equation}\label{eqn:u-update}
   \vy^{k+1} = \left(\mA^\top \mA +\rho\mB^{\top}\mB \right)^{-1} \left( \mA^\top \vb + \mB^\top(\rho \vx^{k+1}+\vz^k) \right). 
\end{equation}
Under the condition of $\text{ker}(\mA)\cap\text{ker}(\mB)=\{\vzero\}$, the solution is unique.

\begin{algorithm} 
\caption{The ADMM algorithm for problem \eqref{model:TVhp_rewrite}} \label{alg:ADMM_IR}
\begin{algorithmic}[1] 
\State \textbf{Input:} Initialize $\vx^0 \in \R^{2n^2}, \vy^0 \in \mathbb{R}^{n^2}$
    \For{$k=1,2,\cdots$}
        \State \begin{eqnarray}
        \vx^{k+1}&\in& \text{prox}_{\frac{\lambda}{\rho_k}\varphi}\left(\mB\vy^{k}-\frac{\vz^k}{\rho_k}\right)\notag\\
    \vy^{k+1} &=& \left(\mA^\top \mA +\rho_k\mB^{\top}\mB \right)^{-1} \left( \mA^\top \vb + \mB^\top(\rho_k \vx^{k+1}+\vz^k) \right) \notag\\
    \vz^{k+1} &=& \vz^k + \rho_k (\vx^{k+1} - \mB\vy^{k+1})\quad \text{and}\quad \rho_{k+1}=\sigma\rho_k \;\;\text{with}\;\;\sigma>1\notag
\end{eqnarray}
    \EndFor
\State \textbf{Output:} $\vy^\infty$
\end{algorithmic}
\end{algorithm}

\subsection{Numerical Experiments: Image Restoration}
In this section, we present numerical experiments to demonstrate the performance of Algorithm~\ref{alg:ADMM_IR} for problem~\eqref{model:TVhp_rewrite}, denoted as $\textbf{h2-ADMM}$, and verify the effectiveness of the proposed method. We evaluate the proposed method on widely-used datasets by comparing it with Total-Variation baselines. In the following, we will first introduce the experimental settings, then we will show the qualitative and quantitative comparison results.

\subsubsection{Experimental Settings}

In this subsection, we introduce the details of the used datasets, baseline models, evaluation metrics, and implementation specifics for our experiments.

We assess the performance of our restoration algorithm using two metrics: the Peak-Signal-to-Noise Ratio (PSNR) and the Structural SIMilarity (SSIM) index. For a restored image, say $\tilde{\vy}$, its PSNR value is defined as
$$
\text{PSNR}:= 10 \log_{10}  \frac{255^2 n}{\|\tilde{\vy}-\vy\|^2},
$$
where $n$ is the number of pixels in $\vy$. A higher PSNR value suggests better image quality, indicating reduced reconstruction error. The SSIM index, evaluating the visual impact of luminance, contrast, and structure of an image, is given by:
$$
\text{SSIM} := \frac{(2 \mu_{\mathbf{u}} \mu_{\tilde{\mathbf{u}}} + c_1)(2 \sigma_{\mathbf{u} \tilde{\mathbf{u}}} + c_2)}{(\mu_{\mathbf{u}}^2 + \mu_{\tilde{\mathbf{u}}}^2 + c_1)(\sigma_{\mathbf{u}}^2 + \sigma_{\tilde{\mathbf{u}}}^2 + c_2)},
$$
where \( \mu_{\mathbf{u}} \) and \( \mu_{\tilde{\mathbf{u}}} \) are the average pixel values of \( \mathbf{u} \) and \( \tilde{\mathbf{u}} \), respectively, \( \sigma_{\mathbf{u}} \) and \( \sigma_{\tilde{\mathbf{u}}} \) are the variances of \( \mathbf{u} \) and \( \tilde{\mathbf{u}} \), and \( \sigma_{\mathbf{u} \tilde{\mathbf{u}}} \) is the covariance of \( \mathbf{u} \) and \( \tilde{\mathbf{u}} \). The constants \( c_1 \) and \( c_2 \) are included to maintain stability when denominators are close to zero. The SSIM index ranges from $-1$ to $1$, with higher values indicating higher similarity and thus better quality.

For our numerical experiments, we assume $\|\mA\|=1$ for the blurring kernel $\mA$ in the degraded model \eqref{model:restoration}. The test images are sourced from the USC-SIPI database (\url{http://sipi.usc.edu/database/}). The blurred and noisy observation $\vb$ is obtained by the image using two methods. The first method involves {pure blurring} the image without the addition of noise. In the second method {blurring with Gaussian noise}, blurring is applied using the Matlab function $\texttt{fspecial('average', hsize)}$, where $\texttt{hsize}$ is the size of the filter, followed by adding Gaussian noise with standard derivation $\sigma$, for example, $\frac{20}{51}$. Alternative blurring techniques may also be employed using different configurations of $\texttt{fspecial}$, such as $\texttt{fspecial('gaussian', hsize, sigma)}$. 

We approximate the original image by solving the optimization model \eqref{model:TVhp_rewrite} using Algorithm~\ref{alg:ADMM_IR}. To evaluate the effectiveness of our method, we compare our method with the Total-Variation image restoration model, specifically the Proximal Alternating Predictor-Corrector (PAPC) algorithm in \cite{Chen-Huang-Zhang:IP:2013,Li-Shen-Xu-Zhang:AiCM:15,Micchelli-Shen-Xu:IP-11}. The PAPC algorithm is designed to minimize the sum of two proper lower-semi-continuous convex functions, formulated as follows:
\begin{equation}\label{eqn:PAPC}
    \vy^\star = \argmin_{\vy\in \R^{n^2}} \;\left(f_1 \circ \mB\right)(\vy)+f_2(\vy)
\end{equation}
where $f_2$ is differentiable on $\R^{n^2}$ with a $1/\beta$-Lipschitz continuous gradient for some $\beta\in (0, +\infty)$. Let $\lambda_{max}(\mB\mB^\top)$ denote the largest eigenvalue of $\mB\mB^\top$. Then the details of the PAPC algorithm are described in Algorithm~\ref{alg:PAPC}.

\begin{algorithm} 
\caption{PAPC algorithm for problem \eqref{eqn:PAPC} (\cite[Algorithm 4]{Chen-Huang-Zhang:IP:2013})} \label{alg:PAPC}
\begin{algorithmic}[1] 
\State \textbf{Input:} Initialize $\vy^0 \in \mathbb{R}^{n^2}$, $\vx^0\in \R^{2n^2}$, $0<\lambda\le 1/\lambda_{max}(\mB\mB^\top)$, $0<\gamma<2\beta$.
    \For{$k=1,2,\cdots$}
        \State \begin{eqnarray}
    \vy^{k+\frac{1}{2}} &=& \vy^k - \gamma \nabla f_2(\vy^k),\notag\\
    \vx^{k+1} &=& \left(I-\text{prox}_{\frac{\gamma}{\lambda}f_1}\right)\left(\mB\vy^{k+\frac{1}{2}}+(I-\lambda \mB\mB^\top)\vx^k\right),\notag\\
    \vy^{k+1} &=& \vy^{k+\frac{1}{2}} - \lambda \mB^\top \vx^{k+1}\notag
\end{eqnarray}
    \EndFor
\State \textbf{Output:} $\vy^\infty$
\end{algorithmic}
\end{algorithm}

By incorporating the PAPC algorithm (Algorithm~\ref{alg:PAPC}) into the isotropic TV restoration model \eqref{model:ITV} (denoted as \textbf{TV-PAPC}), we specify $f_1$ as the isotropic total variation norm, $\|\cdot\|_{\text{ITV}}$, defined in Equation \eqref{def:ITV}. Furthermore, we define $f_2(\vy)$ as $\frac{1}{2\lambda}\|\mA\vy - \vb\|_2^2$, capturing the data fidelity term.

All numerical experiments were performed on a desktop equipped with an Intel i7-7700 CPU (4.20GHz) running MATLAB 9.8 (R2020a). We evaluated the performance of the algorithm across various model parameters $\lambda$, in terms of PSNR, SSIM, and CPU time. The algorithm's iteration terminates if the number of iterations reaches a preset maximum, the PSNR drops significantly, or the relative change in successive updates $\|\vy^{k+1}-\vy^k\|/\|\vy^k\|\le 5 \times 10^{-7}$, where $\vy^{k+1}$ and $\vy^k$ are two successive updates of the algorithm.

\subsubsection{Qualitative and Quantitative Comparisons}
The results presented in this subsection provide a comprehensive comparison between the h2-ADMM and TV-PAPC methods across various settings. Each method's performance is evaluated based on the PSNR and SSIM indices, which quantify the restoration quality in terms of error reduction and structural similarity, respectively. The execution time is also considered to assess the efficiency of each algorithm.

Image deblurring experiments are performed on two natural images Barbara and Finger print to verify the effectiveness of the proposed method. The test images are shown in Figure~\ref{fig:orig} For each image, two groups of experiments are carried out, including deblurring and simultaneous deblurring and denoising. 

\begin{figure}[htbp!]
\centering
% \resizebox{\textwidth}{!}
% {
\begin{tabular}{cccc}
\includegraphics[width=0.33\textwidth]{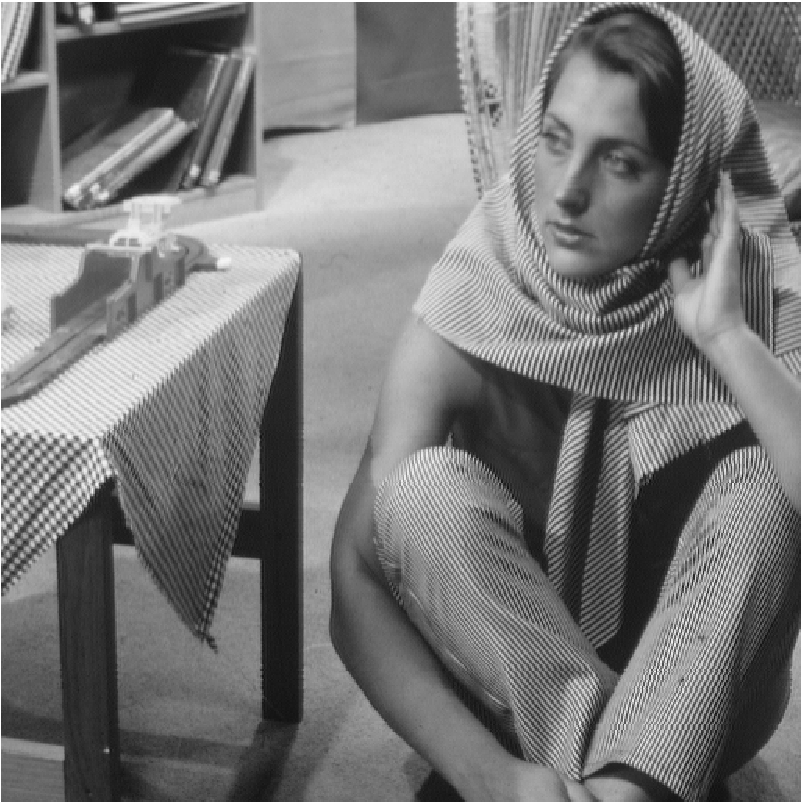}&
\includegraphics[width=0.33\textwidth]{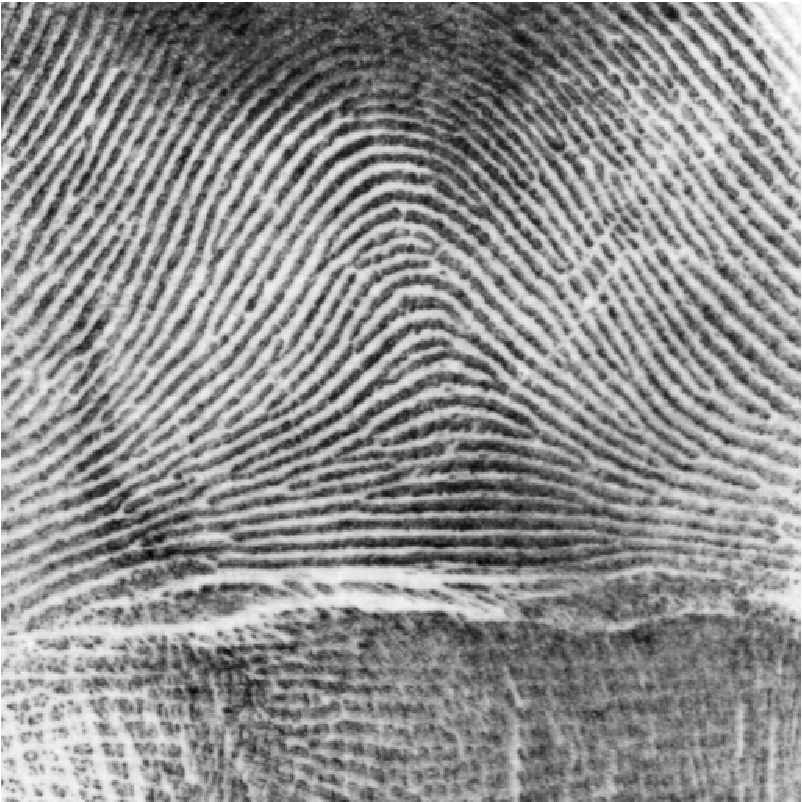}\\
(a)&(b)
\end{tabular}
% }
\caption{Original images: (a) Barbara, (b) Finger Print.}
\label{fig:orig}
\end{figure}

The deblurring experiment results, as shown in Table~\ref{table:pureblurring} and Figure~\ref{fig:deblur_nonoise}, remarkably highlight the superior performance of the h2-ADMM method in terms of both image quality and computational efficiency. This methodological distinction becomes evident through the significantly higher PSNR and SSIM values achieved by h2-ADMM across various image types and blurring filters.

For example, when comparing the results from the ``Barbara" and ``F.print" images under different filter settings, h2-ADMM consistently achieves near-perfect SSIM values (approaching 1.000) and extraordinarily high PSNR values (exceeding 85 in cases using Gaussian blurring), suggesting an almost flawless restoration in visual terms. The execution time for h2-ADMM also significantly outperforms TV-PAPC. We observed that the algorithm stops after only a few iterations with h2-ADMM, a crucial factor in practical applications where processing speed is as vital as output quality.

We note that Total Variation methods in image restoration primarily aim to minimize the sum of the differences in pixel values along the horizontal and vertical directions. One potential drawback is that they can sometimes overly smooth the image, leading to the loss of important details and textures. This effect is particularly noticeable in regions where fine patterns or subtle textures are present. In contrast, our proposed model \eqref{model:TVhp_rewrite} addresses this by normalizing the 2-dimensional vector of each pixel's horizontal and vertical differences. This normalization process primarily focuses on the relative ratio of the horizontal to vertical differences, rather than their absolute values. By minimizing the $\ell_1$ norm of the normalized vector, our model promotes the preservation of more details, effectively maintaining the integrity of subtle textures and fine patterns in the image.
\begin{figure}[htbp!]
\centering
% \resizebox{\textwidth}{!}
% {
\begin{tabular}{cccc}
\includegraphics[width=0.3\textwidth]{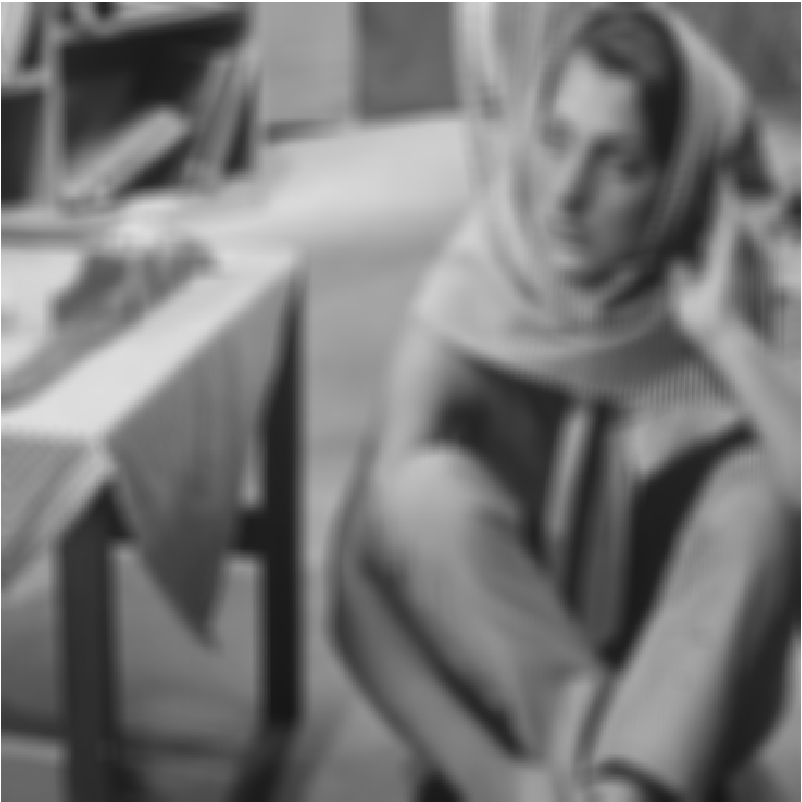}&
\includegraphics[width=0.3\textwidth]{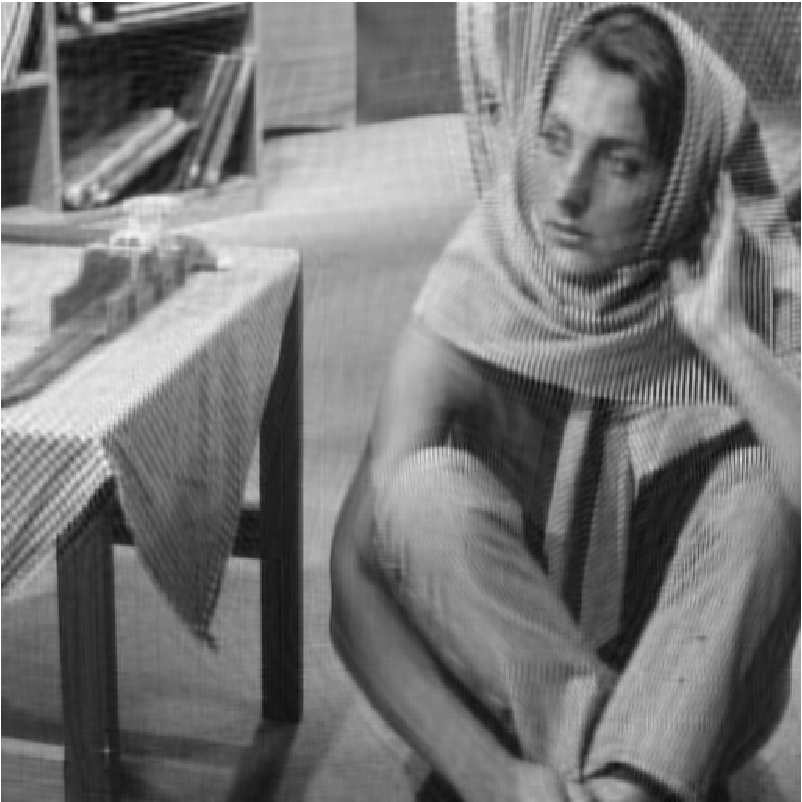}&
\includegraphics[width=0.3\textwidth]{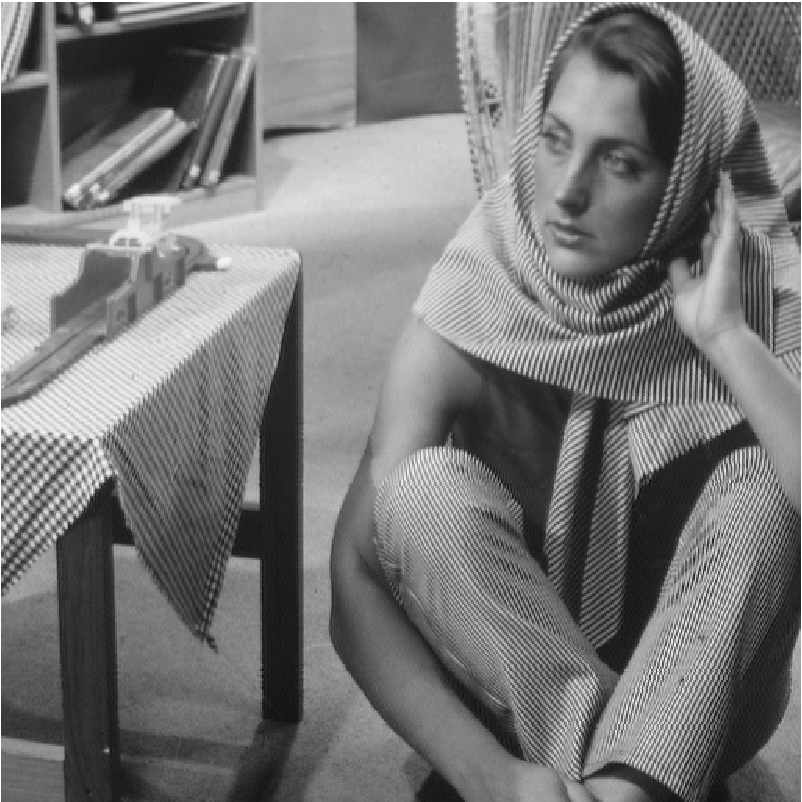}\\
\includegraphics[width=0.3\textwidth]{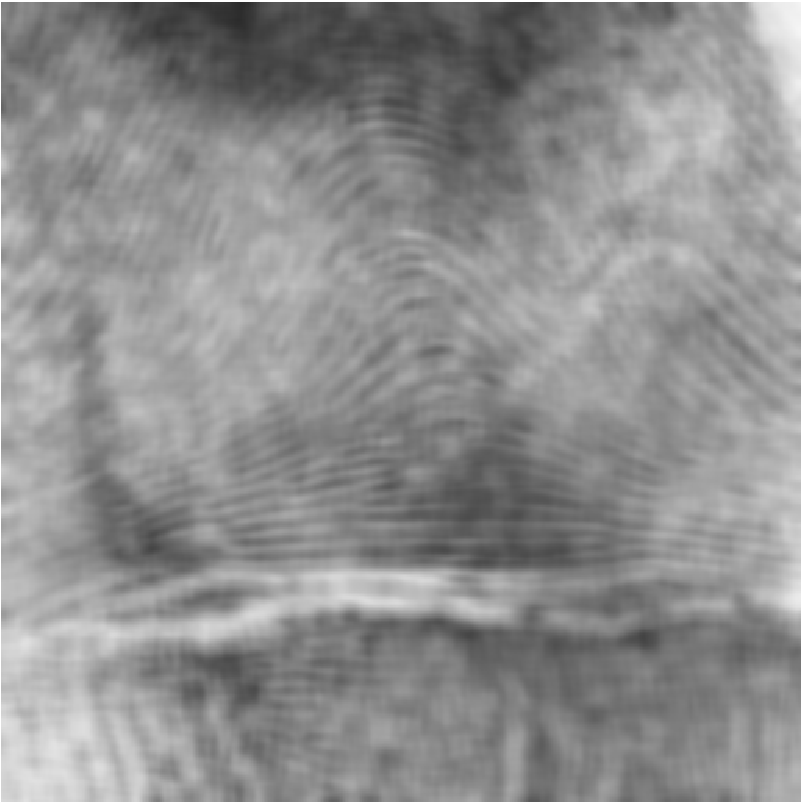}&
\includegraphics[width=0.3\textwidth]{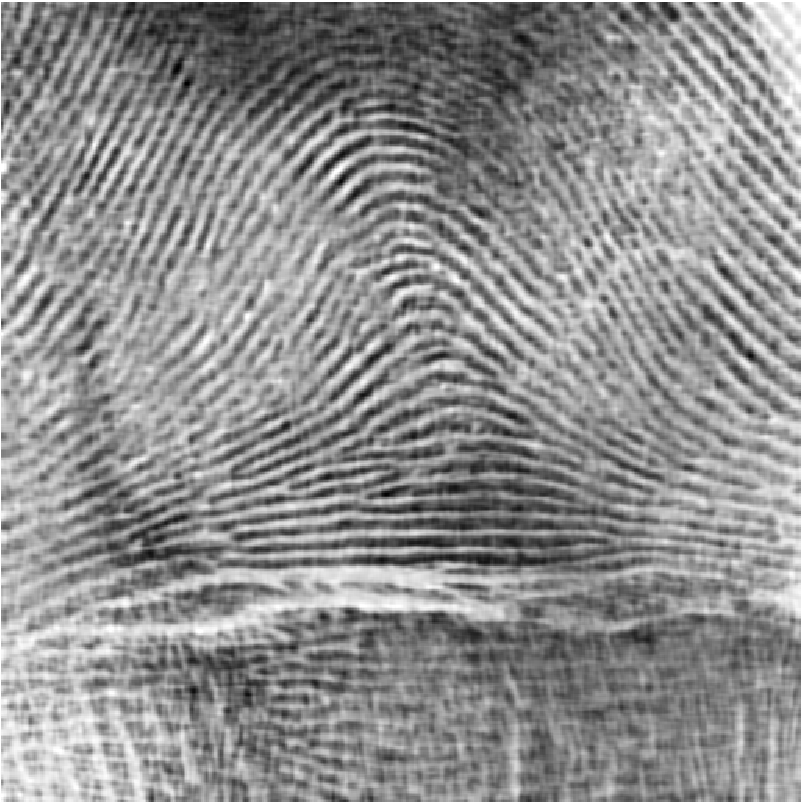}&
\includegraphics[width=0.3\textwidth]{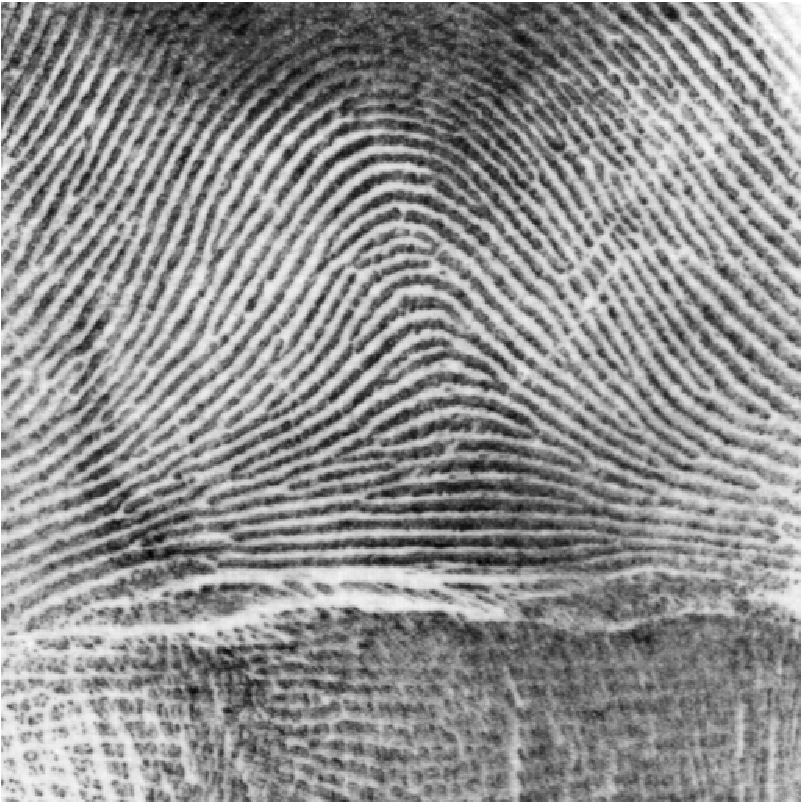}\\
(a)&(b)&(c)
\end{tabular}
% }
\caption{Quantitative comparisons of deblurring results: (a) degradation without noise, (b) restoration using the TV-PAPC method, and (c) restoration using the h2-ADMM method.}
\label{fig:deblur_nonoise}
\end{figure}

\begin{table}[ht]
\caption{Restoration results for blurred observations}
\centering
\resizebox{\textwidth}{!}
{
\begin{tabular}{cc|ccccc|cccc}
\noalign{\hrule height 1pt}
&  &  \multicolumn{5}{c|}{h2-ADMM} & \multicolumn{4}{c}{TV-PAPC} \\ 
\hline
Image &  blurring filter & $\lambda_{\text{h2-ADMM}}$ & $\rho_{\text{h2-ADMM}}$     & PSNR     & SSIM  & Time  &  $\lambda_{\text{TV-PAPC}}$     & PSNR    & SSIM  & Time  \\ 
\noalign{\hrule height 1pt}
\multirow{5}{*}{Barbara} & $\texttt{(`average', 5)}$ & 2e-02 & 1e-08 & 58.627 & 0.998 &   2.8125 & 1e-02 & 29.412 & 0.793 &  24.0156  \\
&$\texttt{(`average', 9)}$& 2e-02 & 1e-08 & 55.867 & 0.997 &   2.8125 & 1e-02 & 25.730 & 0.620 &  25.6250 \\
&$\texttt{(`average', 13)}$& 2e-02 & 1e-08 & 48.255 & 0.986 &   2.6406 & 1e-02 & 24.705 & 0.532 &  26.0469 \\
&$\texttt{(`gaussian', 13, 1)}$& 1e-01 & 1e-08 & 85.439 & 1.000 &   2.7656 & 1e-02 & 33.950 & 0.888 &  25.5625 \\
&$\texttt{(`gaussian', 25, 1)}$& 1e-01 & 1e-08 & 85.439 & 1.000 &   2.9063 & 1e-02 & 33.950 & 0.888 &  26.9219 \\
\hline
\multirow{5}{*}{F.print} & $\texttt{(`average', 5)}$ & 5e-02 & 1e-08 & 64.613 & 1.000 &   2.9844 & 1e-03 & 35.869 & 0.984 &  23.0625 \\
&$\texttt{(`average', 9)}$& 2e-02 & 1e-08 & 61.154 & 1.000 &   2.8125 & 1e-03 & 29.126 & 0.931 &  24.9844  \\
&$\texttt{(`average', 13)}$& 1e-02 & 1e-08 & 59.533 & 1.000 &   2.6094 & 1e-03 & 22.176 & 0.772 &  25.0156 \\
&$\texttt{(`gaussian', 13, 1)}$& 1e-01 & 1e-08 & 84.686 & 1.000 &   2.7656 & 1e-03 & 43.778 & 0.997 &  25.0938 \\
&$\texttt{(`gaussian', 25, 1)}$& 1e-01 & 1e-08 & 84.686 & 1.000 &   2.7188 & 1e-03 & 43.778 & 0.997 &  24.4844 \\
\noalign{\hrule height 1pt}
\end{tabular}
}
\label{table:pureblurring}
\end{table}

\begin{table}[ht]
\caption{Restoration results for blurred observations with noise}
\centering
\resizebox{\textwidth}{!}
{
\begin{tabular}{cc|ccccc|cccc}
\noalign{\hrule height 1pt}
&  &  \multicolumn{5}{c|}{h2-ADMM} & \multicolumn{4}{c}{TV-PAPC} \\ 
\hline
Image &  blurring filter and noise & $\lambda_{\text{h2-ADMM}}$ & $\rho_{\text{h2-ADMM}}$     & PSNR     & SSIM  & Time  &  $\lambda_{\text{TV-PAPC}}$     & PSNR    & SSIM  & Time  \\ 
\noalign{\hrule height 1pt}
\multirow{10}{*}{Barbara} & $\texttt{(`average', 5)}$, $\sigma=\frac{21}{50}$ & 0.020 & 0.010 & 29.660 & 0.744 &  4.8438 & 0.010 & 28.636 & 0.749 &  8.3906 \\
&$\texttt{(`average', 9)}$, $\sigma=\frac{21}{50}$& 0.020  & 0.010 & 26.760 & 0.607 &  6.3281& 0.010 & 25.277 & 0.581 & 10.5156\\
&$\texttt{(`average', 13)}$, $\sigma=\frac{21}{50}$& 0.020  & 0.010 & 25.300 & 0.546 &  6.0781
& 0.010 & 24.310 & 0.497 & 10.0000\\
&$\texttt{(`average', 5)}$, $\sigma=1$& 0.100  & 0.010 & 27.571 & 0.627 &  4.8750& 0.010 & 27.414 & 0.604 &  9.1875\\
&$\texttt{(`average', 9)}$, $\sigma=1$& 0.100   & 0.010 &25.042 & 0.522 &  5.2813&0.010 & 24.906 & 0.504 & 10.1875\\
&$\texttt{(`average', 13)}$, $\sigma=1$& 0.100   & 0.010 &24.333 & 0.468 &  5.8594&0.010 & 24.101 & 0.454 &  9.8125\\
&$\texttt{(`gaussian', 13, 1)}$, $\sigma=\frac{21}{50}$& 0.100 &   0.005 & 33.308 & 0.860 &  6.1875& 0.010 & 32.580 & 0.857 &  9.9219\\
&$\texttt{(`gaussian', 25, 1)}$, $\sigma=\frac{21}{50}$& 0.100 &   0.005 & 33.238 & 0.860 &  5.8438& 0.010 & 32.580 & 0.857 & 10.8750\\
&$\texttt{(`gaussian', 13, 1)}$, $\sigma=1$& 0.100 &  0.010 & 30.752 & 0.758 &  4.8125& 0.010 & 30.519 & 0.720 & 10.0469\\
&$\texttt{(`gaussian', 25, 1)}$, $\sigma=1$& 0.100 & 0.010 & 30.752 & 0.758 &   5.1719 & 0.010 & 30.519 & 0.720 &  11.2500 \\
\hline
\multirow{10}{*}{F.print} & $\texttt{(`average', 5)}$, $\sigma=\frac{21}{50}$ & 0.050 & 0.010 & 33.500 & 0.971 & 3.9844 & 0.010 & 33.116 & 0.968 & 9.8594\\
&$\texttt{(`average', 9)}$, $\sigma=\frac{21}{50}$& 0.020 & 0.010 & 29.349 & 0.930 &   3.9844 & 0.001 & 27.341 & 0.897 &  10.5156 \\
&$\texttt{(`average', 13)}$, $\sigma=\frac{21}{50}$& 0.010 & 0.010 & 25.171 & 0.859 &   3.5469 & 0.001 & 22.045 & 0.761 &  26.5938 \\
&$\texttt{(`average', 5)}$, $\sigma=1$& 0.200 & 0.010 & 31.197 & 0.951 &   3.4375 & 0.010 & 30.755 & 0.946 &   2.0469 \\
&$\texttt{(`average', 9)}$, $\sigma=1$& 0.100 & 0.010 & 27.240 & 0.889 &   3.6719 & 0.005 & 26.332 & 0.869 &  11.0000 \\
&$\texttt{(`average', 13)}$, $\sigma=1$& 0.030 & 0.010 & 23.003 & 0.777 &   3.6094 & 0.005 & 21.488 & 0.719 &  26.1719 \\
&$\texttt{(`gaussian', 13, 1)}$, $\sigma=\frac{21}{50}$& 0.100 & 0.010 & 38.557 & 0.990 &   3.4688 & 0.001 & 38.511 & 0.990 &   2.6094 \\
&$\texttt{(`gaussian', 25, 1)}$, $\sigma=\frac{21}{50}$& 0.100 & 0.010 & 38.557 & 0.990 &   3.4375 & 0.001 & 38.511 & 0.990 &   2.3906\\
&$\texttt{(`gaussian', 13, 1)}$, $\sigma=1$& 0.500 & 0.010 & 35.946 & 0.983 &   3.3125 & 0.050 & 35.431 & 0.981 &   2.2031 \\
&$\texttt{(`gaussian', 25, 1)}$, $\sigma=1$& 0.500 & 0.010 & 35.946 & 0.983 &   3.5469 & 0.050 & 35.431 & 0.981 &   2.3906 \\
\noalign{\hrule height 1pt}
\end{tabular}
}
\label{table:PSNR}
\end{table}

For simultaneous deblurring and denoising, in examining the performance across different filters and noise levels, it is evident that both the H2-ADMM and TV-PAPC methods generally produce comparable results in terms of PSNR and SSIM values, as shown in Tabel~\ref{table:PSNR} and Figure~\ref{fig:deblur_noise}.

When considering the efficiency of the algorithms, similar observations were made in cases of pure blurring. The h2-ADMM algorithm reaches a satisfying result after only a few iterations, hence consistently outperforming TV-PAPC in terms of processing time across almost all test cases.

In terms of restoration quality, h2-ADMM achieves slightly higher PSNR values and maintains comparable or superior SSIM values relative to TV-PAPC. This indicates the effectiveness of the proposed model \eqref{model:TVhp_rewrite} and the aligned h2-ADMM algorithm.

We note that compared to the pure deblurring case shown above, the superiority of the h2-ADMM is not particularly evident in cases involving simultaneous deblurring and denoising. As previously mentioned, the proposed model tends to preserve more detail by promoting the normalized 2-dimensional difference vector for pixels with a large relative ratio between its two entries. This approach may inadvertently preserve more details along with some heavy noise. This is likely the reason our method does not exhibit strong superiority in scenarios involving deblurring with noise. Therefore, we rate our method as being well-suited for deblurring with no noise or only slight noise. Overall, our comparative analysis underscores the efficacy of the h2-ADMM algorithm in delivering high-quality image restoration with greater efficiency.

\begin{figure}[htbp!]
\centering
% \resizebox{\textwidth}{!}
% {
\begin{tabular}{cccc}
\includegraphics[width=0.3\textwidth]{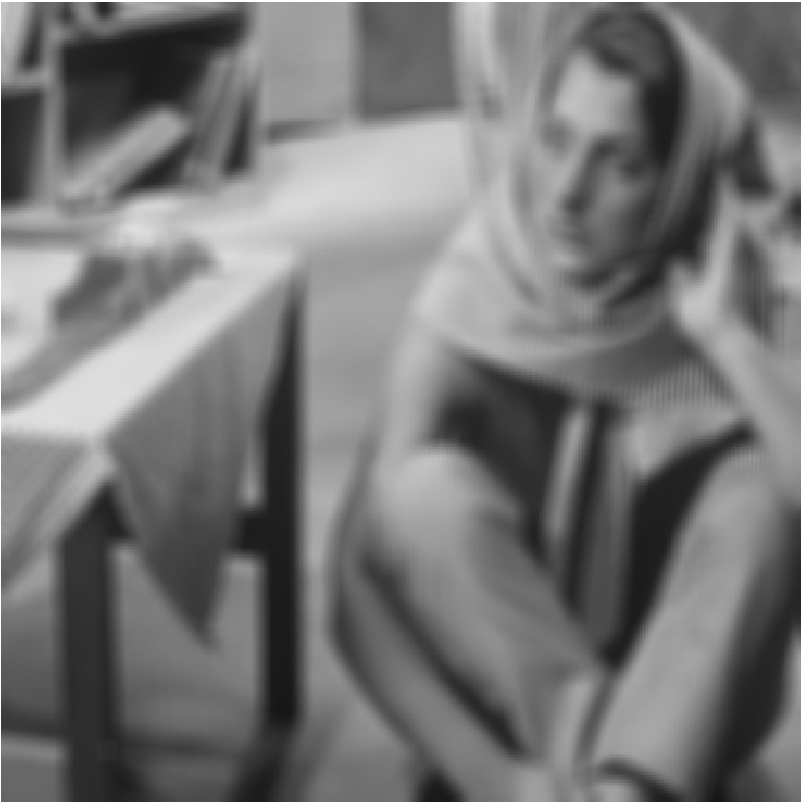}&
\includegraphics[width=0.3\textwidth]{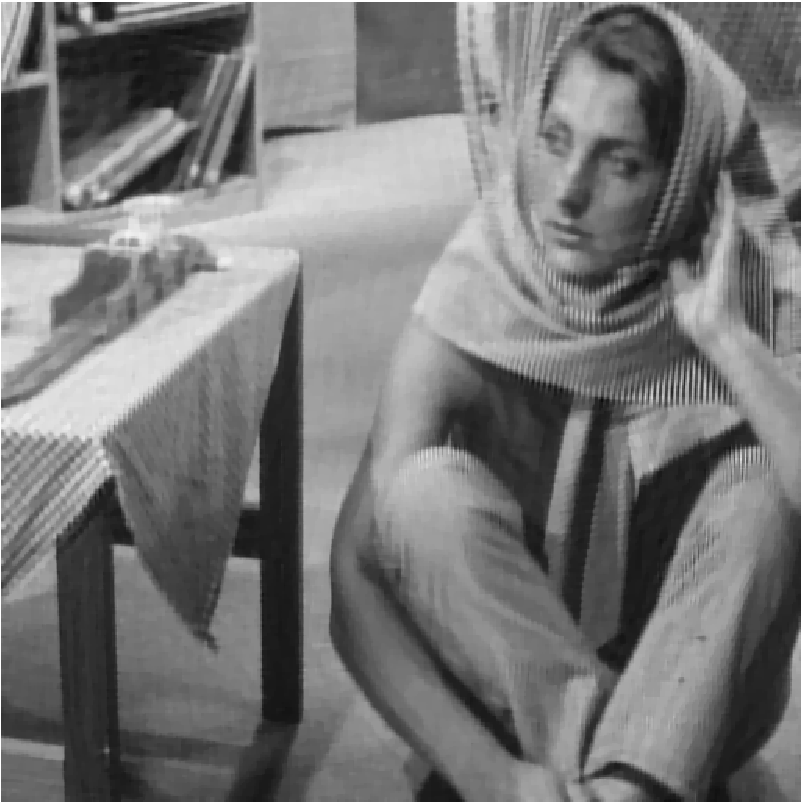}&
\includegraphics[width=0.3\textwidth]{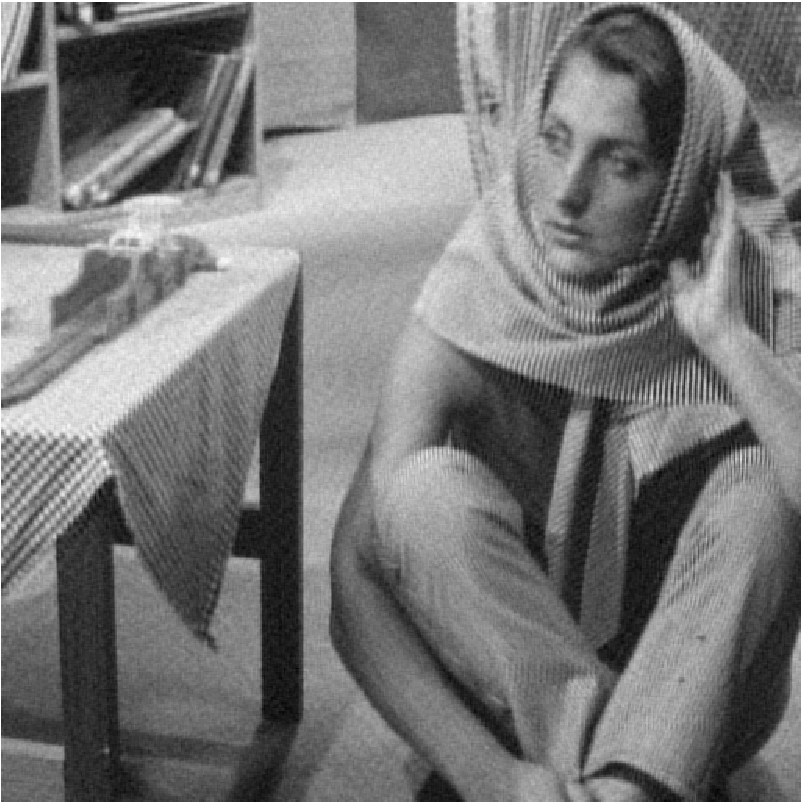}\\
\includegraphics[width=0.3\textwidth]{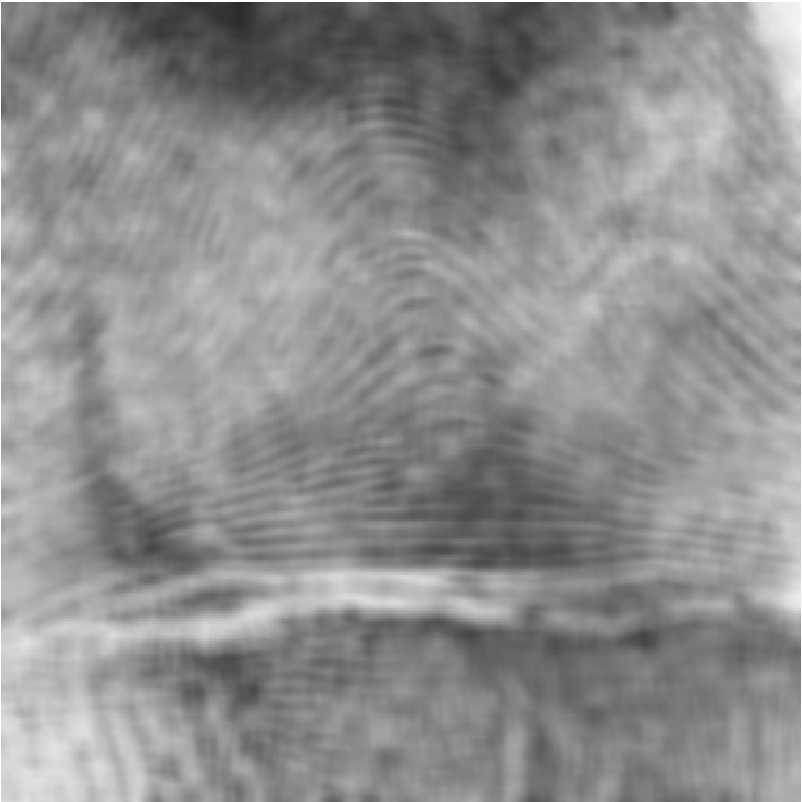}&
\includegraphics[width=0.3\textwidth]{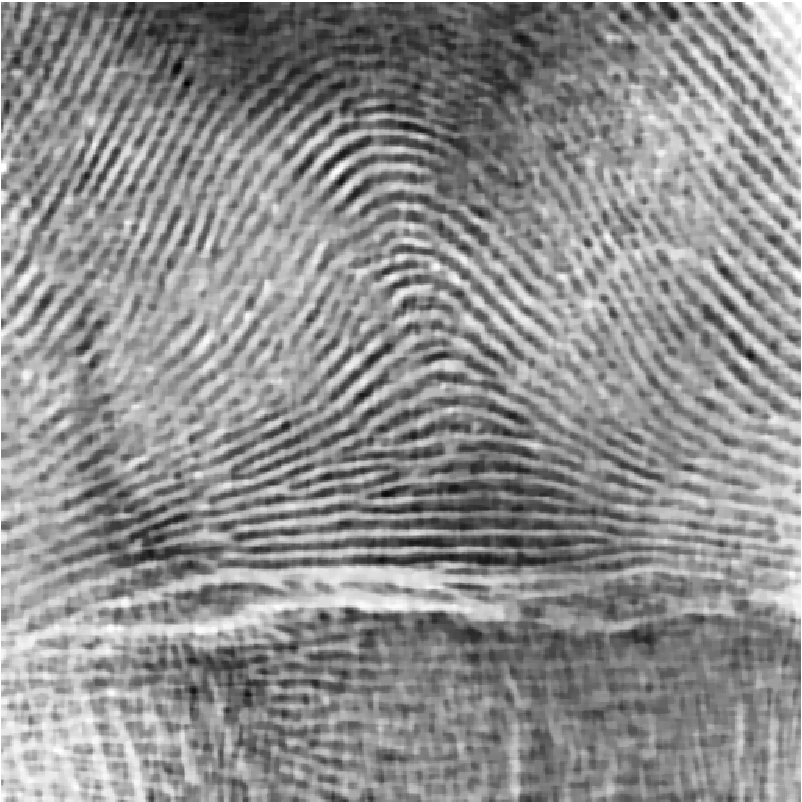}&
\includegraphics[width=0.3\textwidth]{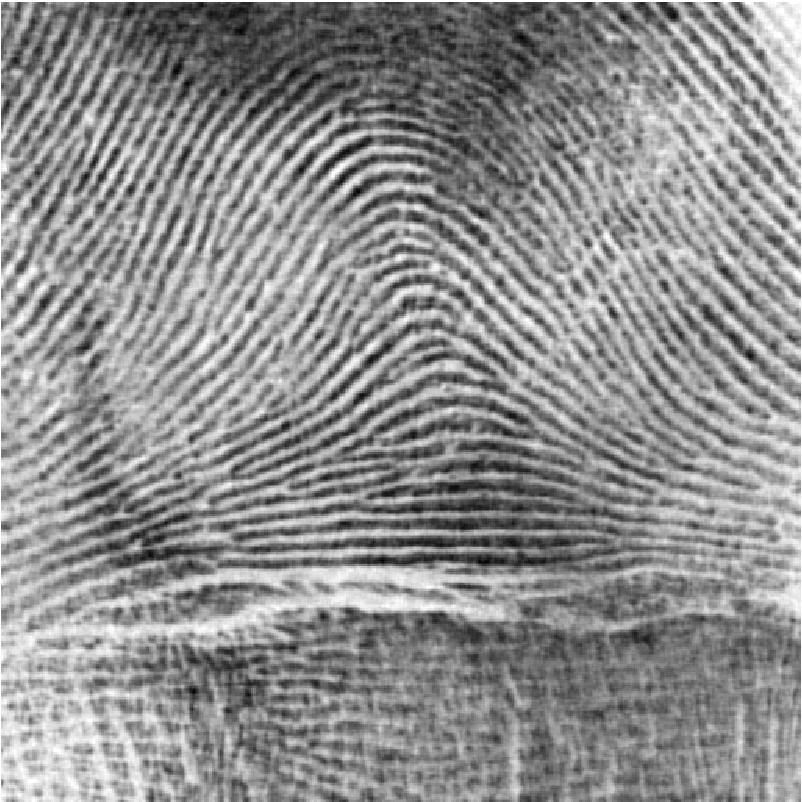}\\
(a)&(b)&(c)
\end{tabular}
% }
\caption{Quantitative comparisons of deblurring results: (a) degradation with noise, (b) restoration using the TV-PAPC method, and (c) restoration using the h2-ADMM method.}
\label{fig:deblur_noise}
\end{figure}

\section{Conclusion}
This study introduced an advanced approach for integrating the proximity operators of $(\ell_1/\ell_2)^2$ and $\ell_1/\ell_2$ into the frameworks of APG and ADMM, enhancing the efficacy of sparse signal and image recovery methods. Our detailed exploration of the mathematical properties of these two scale and signed permutation invariant sparsity-promoting functions, particularly their behavior under the proximity operator, has provided new insights into their potential to improve recovery accuracy and computational efficiency in sparse signal recovery and image restoration. The theoretical discussions on convergence properties reinforce the robustness and reliability of the proposed methods. Future work will focus on extending the application of these techniques to broader areas of medical image processing and machine learning where sparsity is a critical component. Additionally, exploring the integration of other scale and signed permutation invariant functions into recovery algorithms could uncover new pathways to further enhance performance.

\section*{Acknowledgement}
This research was supported by the National Science Foundation under grant DMS-2208385.

\bibliographystyle{siam}
\bibliography{shen_jia_2023.bib}

\end{document}